\documentclass[twoside,12pt, letterpaper,reqno]{amsart}
\usepackage[dvips]{graphics}
\usepackage{times}
\usepackage{mathrsfs}
\usepackage[T1]{fontenc}
\usepackage{latexsym}
\usepackage{epsfig}
\usepackage{hyperref}
\hypersetup{
    pdfborder = {0 0 0}
}
\usepackage{amsmath,amsfonts,amsthm,amssymb,amscd}
\usepackage{pstricks}
\usepackage[myheadings]{fullpage}

\usepackage{mathdots}

\bibdata{Bibliography.bib}

\newcommand{\vep}{\varepsilon}
\newcommand{\eset}{\varnothing}
\newcommand{\ul}{\underline}
\DeclareMathOperator{\supp}{supp}

\newcommand{\sumint}[1]{\sum_p \frac{\log p}{\sqrt{p}} \left(\frac{8d}{p}\right) \widehat{#1} \left(\frac{\log p}{\log X} \right)}
\newcommand{\ml}[1]{\mathcal{#1}}
\newcommand{\ba}{\begin{align}}
\newcommand{\eal}{\end{align}}
\newcommand{\op}{\operatorname}

\newtheorem{theorem}{Theorem}[section]
\newtheorem{conjecture}[theorem]{Conjecture}

\newtheorem{lemma}[theorem]{Lemma}

\newtheorem{proposition}[theorem]{Proposition}

\theoremstyle{definition}
\newtheorem{definition}[theorem]{Definition}
\newtheorem{rmk}[theorem]{Remark}

\theoremstyle{theorem}

\newcommand{\kommentar}[1]{}

\newcommand{\zeev}{Ze\'{e}v}

\newcommand\be{\begin{equation}}
\newcommand\ee{\end{equation}}
\newcommand\bea{\begin{eqnarray}}
\newcommand\eea{\end{eqnarray}}
\newcommand\bi{\begin{itemize}}
\newcommand\ei{\end{itemize}}
\newcommand\ben{\begin{enumerate}}
\newcommand\een{\end{enumerate}}
\newcommand\bc{\begin{center}}
\newcommand\ec{\end{center}}

\newcommand{\R}{\ensuremath{\mathbb{R}}}
\newcommand{\C}{\ensuremath{\mathbb{C}}}

\newcommand{\N}{\mathbb{N}}

\newtheorem{thm}{Theorem}[section]

\newtheorem{prop}[thm]{Proposition}

\theoremstyle{definition}
\newtheorem{defi}[thm]{Definition}
\newtheorem{rek}[thm]{Remark}

\newcommand{\twocase}[5]{#1 \begin{cases} #2 & \text{{\rm #3}}\\ #4
&\text{{\rm #5}} \end{cases}   }

\numberwithin{equation}{section}
\begin{document}


\baselineskip=17pt



\title[$n$-level densities of Dirichlet $L$-functions]{The $n$-level densities of low-lying zeros of quadratic Dirichlet $L$-functions}

\author[J. Levinson]{Jake Levinson}\email{jakelev@umich.edu}
\address{Department of Mathematics, University of Michigan, Ann Arbor, MI 48109}

\author[S. J. Miller]{Steven J. Miller}\email{sjm1@williams.edu, Steven.Miller.MC.96@aya.yale.edu}
\address{Department of Mathematics and Statistics, Williams College,
Williamstown, MA 01267}



\date{\today}

\thanks{
The Mathematica files are available online at\\ \tiny \url{http://web.williams.edu/Mathematics/sjmiller/public_html/math/papers/jakel/FourierIdentity.tar}.\normalsize}

\begin{abstract} Previous work by Rubinstein \cite{Rub} and Gao \cite{Gao} computed the $n$-level densities for families of quadratic Dirichlet $L$-functions for test functions  $\widehat{f}_1, \dots, \widehat{f}_n$ supported in $\sum_{i=1}^n |u_i| < 2$, and showed agreement with random matrix theory predictions in this range for $n \le 3$ but only in a restricted range for larger $n$. We extend these results and show agreement for $n \le 7$, and reduce higher $n$ to a Fourier transform identity. The proof involves adopting a new combinatorial perspective to convert all terms to a canonical form, which facilitates the comparison of the two sides.
\end{abstract}

\maketitle

\tableofcontents

\section{Introduction}

Assuming the Generalized Riemann Hypothesis (GRH), the non-trivial zeros of $L$-functions lie on the line $\Re{s} = 1/2$. The distribution of these zeros govern the behavior of a variety of problems, ranging from the distribution of primes in arithmetic progressions to the size of the class number to the geometric rank of the Mordell-Weil group of elliptic curves, among others \cite{CI,Da,Go,GZ,RubSa}. In many instances we need to know more than just the fact that the zeros lie on the line, but additionally how they are distributed on the line.

One of the most successful approaches to modeling these zeros is through Random Matrix Theory. Originally arising in statistical investigations \cite{Wis}, the subject flourished in the 1950s and 1960s with the work of Wigner \cite{Wig1, Wig2, Wig3, Wig4, Wig5}, Dyson \cite{Dy1, Dy2} and others, who applied it to describe the energy levels of heavy nuclei. In the 1970s, Montgomery and Dyson \cite{Mon} noticed that the 2-level correlation of zeros of the Riemann zeta function matched those of the Gaussian Unitary Ensemble (GUE); see \cite{Ha,FirMi} for more on the history. Since then Random Matrix Theory has made precise statements about the main term in the behavior of numerous statistics involving zeros of $L$-functions \cite{Con,KeSn1,KeSn2,KeSn3}.

While the limiting behavior of $n$-level correlations of a single $L$-function have been shown to agree (for suitable test functions) with the scaling limit of the GUE \cite{Hej,Mon,RS}, the behavior near the central point is different for different $L$-functions, and depends on the arithmetic of the form (for example, the order of vanishing of $L$-functions attached to elliptic curves is conjecturally equal to the rank of the Mordell-Weil group). To study these low-lying zeros, Katz and Sarnak \cite{KaSa1,KaSa2} introduced a different statistic, the $n$-level density, defined as follows. Assuming GRH, the non-trivial zeros of an $L$-function $L(s,g)$ are $1/2 + i\gamma_g^{(j)}$ with $\gamma_g^{(j)}$ real, where $\cdots \le \gamma_g^{(-2)} \le \gamma_g^{(-1)} \leq\gamma_g^{(1)}\leq\gamma_g^{(2)} \leq\cdots$ if the sign of the functional equation is even (if it is odd, there is an extra zero: $\gamma_g^{(0)} = 0$). The $n$-level density for a finite family of $L$-functions $\mathcal{G}$ is
\begin{eqnarray}
D^{(n)}(\mathcal{G};f) \ := \ \frac{1}{|\mathcal{G}|} \sum_{g\in \mathcal{G}} \sum_{\substack{j_1,\dots, j_n \\ j_i \neq \pm j_k}} f_1\left(\frac{\log R}{2\pi}\gamma_g^{(j_1)}\right)\cdots f_n\left(\frac{\log R}{2\pi}\gamma_g^{(j_n)}\right),
\end{eqnarray} where the $f_i$ are even Schwartz functions whose Fourier transforms  have compact support and $\log R$ is a normalization parameter (essentially the average of the logarithms of the analytic conductors) so that the scaled zeros near the central point have mean spacing 1. The Katz-Sarnak Density Conjecture states that as the conductors tend to infinity the distribution of the scaled zeros near the central point converges to the same limiting distribution as the normalized eigenvalues near 1 of a subgroup of the unitary group $U(N)$ as $N\to\infty$. The corresponding group is typically unitary, symplectic, or orthogonal matrices (or a trivial modification to take into account forced zeros at the central point). There is strong evidence for this conjecture. First, in the function field case the correspondence is clear as the subgroup is the monodromy group. Second, there are now many families of $L$-functions where we can prove agreement for suitably restricted test functions, including Dirichlet $L$-functions, elliptic curves, cuspidal newforms, Maass forms, number field $L$-functions and symmetric powers of ${\rm GL}_2$ automorphic representations, to name a few \cite{AILMZ, AM, DM1, FioMi, FI, Gao, GK, Gu, HM, HR, ILS, KaSa1, KaSa2, Mil1, MilPe, OS1, OS2, RR, Ro, Rub, Ya, Yo}. More generally, in recent work Shin and Templier \cite{ShTe} determined the symmetry type of many families of automorphic forms on ${\rm GL}_n$ over $\mathbb{Q}$, and Due\~nez and Miller \cite{DM2} showed how to understand the low-lying zeros of many compound families of $L$-functions (arising from Rankin-Selberg convolutions) in terms of the behavior of the constituent families.

In this paper we study the low-lying zeros (i.e., those near the central point) of quadratic Dirichlet $L$-functions via the $n$-level density. In his thesis Rubinstein \cite{Rub} showed these agree with the scaling limit of symplectic matrices whenever $\widehat{f}_1, \dots, \widehat{f}_n$ are supported in $\sum_{i=1}^n |u_i| < 1$. Gao \cite{Gao} extended this result in his thesis. It is important to have as large support as possible, as frequently extending the support is related to finer questions about the arithmetic of the family. Interestingly, while Gao was able to compute the number theory side for test functions supported in $\sum_{i=1}^n |u_i| < 2$, he was only able to show agreement with the Katz-Sarnak determinantal expansion for the symplectic ensemble for $n \le 3$.

This created an annoying situation in the literature, where both number theory and random matrix theory had been computed in the regime $\sum_{i=1}^n |u_i| < 2$, but could only be shown to agree in this full range for $n \le 3$. Gao's proof involved using ad hoc Fourier transform identities to match the manageable number of terms present for such small $n$. Unfortunately, the number of summands grows very rapidly with $n$, and this approach becomes impractical for higher $n$.

In this paper, we further extend the agreement between number theory and random matrix theory. Our proof is in two steps. First, we resolve a combinatorial obstruction by rewriting both densities using the same combinatorial perspective: we express the terms of the densities in terms of certain pairs of set partitions. This allows us to show agreement between most of the terms arising in the densities, for any $n$. Second, we reduce the Density Conjecture (in the range $\sum_{i=1}^n |u_i| < 2$) to showing that a term arising in the random matrix theory is the Mobius transform (over the lattice of set partitions) of a corresponding term from number theory. We cannot prove this identity for all $n$, but we use Mobius inversion and properties of Fourier transforms to give it a canonical form that is possible to check with a computation. As an application, we verify it for $n \leq 7$: \\

\begin{thm}\label{thm:mainnlevelagree} Let $f_1, \ldots, f_n$ be even Schwartz functions with the $\widehat{f_i}$ supported in $\sum_{i=1}^n |u_i|$ $<$ $2$. For $n \leq 7$, the Katz-Sarnak Density Conjecture holds for the low-lying zeros of quadratic Dirichlet $L$-functions $\{L(s,\chi_{8d})\}$, where $d \in \N$ is odd and square-free.
\end{thm}

In the above theorem, following Gao \cite{Gao} we restricted the family of quadratic characters. This simplifies the analysis by excluding $\chi_2$, and facilitates applications of Poisson summation in Gao's thesis \cite{Gao}. Note that $\chi_{8d}$ is a real primitive character with even sign (i.e., $\chi_{8d}(-1) = 1$).

We briefly sketch the proof. Both sides are known for $\sum_{i=1}^n |u_i| < 2$ by \cite{Rub,Gao}; the difficulty is showing that the two expressions are equal. We proceed as follows.

\begin{enumerate}

\item We regroup the terms in the random matrix theory in terms of pairs of set partitions $\ul{F}, \ul{G}$, such that $\ul{F}$ refines $\ul{G}$ and each block of $\ul{G}$ is a union of at most two blocks of $\ul{F}$.
\item We do the same to the number theory; this step is more involved because the counting is naturally `backwards' there, so the main step is to switch the order in which the pairs of partitions are counted.
\item We separate the remaining non-matching terms from the rest of the sum, and show that they are all instances of a single Mobius inversion identity. 
\item We (Mobius-)invert the identity and use properties of Fourier transforms to convert all the terms to integrals over $\R^n_{\geq 0}$. We reduce to showing that the integrands are identically equal in the region $u_i > 0$, $\sum_{i=1}^n u_i < 2$.
\item We reduce to a formal polynomial over the subsets of $\{1, \ldots, n\}$, modulo two relations that encode the support restriction: this gives an algorithm for showing the Fourier identity, which we use to verify up to $n = 7$.
\end{enumerate}

\begin{rek} This work is an extension of the first-named author's 2011 senior thesis at Williams College \cite{Lev}. There agreement was shown for $n \le 6$ through a more computational approach. In the course of extending these results and preparing this manuscript, we learned of the work of Entin, Roddity-Gershon and Rudnick \cite{ER-GR}, who are able to show agreement for all $n$. Instead of taking a combinatorial approach, they proceed by going through a function field analogue and using the limit of large finite fields where the hyperelliptic ensemble is shown to have USp statistics. In particular, their results imply that our identity holds for all $n$; it would be interesting to complete the ideas of this paper and derive a purely combinatorial proof of this fact.
\end{rek}

The paper is organized as follows. We assume the reader is familiar with \cite{Rub,Gao}, and we will just quote the number theory and random matrix expansions from these works. In \S\ref{sec:combpreliminaries} we review some notation and derive some combinatorial results which allow us to recast our problem as a related Fourier transform identity. We briefly discuss the obstruction which restricts our theorem to $n \le 7$, and see why the two sides at first look so different. We continue in the next section by recasting the random matrix and number theory expansions to a more amenable form, reducing the problem to the aforementioned Fourier transform identity, which we analyze in \S\ref{sec:fourier_identity_n6}. There we rewrite everything in a more tractable canonical form, and discuss the verification for $n \le 7$, which completes the proof of Theorem \ref{thm:mainnlevelagree}.

\section{Combinatorial preliminaries}\label{sec:combpreliminaries}

The purpose of this section is to set the notation for the subsequent combinatorial analysis, and highlight the technical issues.

\subsection{Set partitions}

We recall some basic properties of set partitions. A {\bf partition} $\ul{F}$ of a finite set $S$ is a collection of subsets $\ul{F} = \{F_1, \ldots, F_k\} \subset \ml{P}(S)$, such that the $F_i$ are nonempty and pairwise disjoint, and $S = \bigcup_{i=1}^k F_i$. The $F_i$ are called the {\bf blocks} of $\ul{F}$ and the number  $k = \nu(\ul{F})$ of blocks is the {\bf length} of $\ul{F}$.
The set of all partitions of a set $S$ is denoted $\Pi(S)$;
when $n \in \N$, by abuse of notation we write $\Pi(n)$ for $\Pi(\{1,\ldots,n\})$.

We partially-order $\Pi(S)$ by partition refinement: $\ul{F} \preceq \ul{G}$ if each block of $\ul{F}$ is contained in some block of $\ul{G}$ (equivalently, each block of $\ul{G}$ is a union of blocks of $\ul{F}$). We write $\ul{O} = \{\{1\}, \ldots, \{n\}\}$ and $\ul{N} = \{\{1, \ldots, n\}\}$ for the minimal and maximal partitions.


We associate to any partially ordered set $P$ the incidence algebra
\be\ml{A}\ =\ \{f : P \times P \to \C \mid f(x,y) = 0 \text{ unless } x \preceq y\},\ee
with pointwise addition and multiplication defined by the convolution $*$:
\be (f \ast g)(x,y)\ =\ \sum_{z \in [x,y]} f(x,z)g(z,y),\ee
where \be [x,y] \ := \ \{z : x \preceq z \preceq y\} \ee is the segment from $x$ to $y$. The multiplicative identity is denoted by $\delta$, where \be \twocase{\delta(x,y) \ = \ }{1}{if $x = y$}{0}{otherwise.} \ee We have the zeta function \be \twocase{\zeta(x,y)\ = \ }{1}{if $x \preceq y$}{0}{otherwise.}\ee We think of multiplication by $\zeta$ as `integration', since
\begin{align}
(\zeta \ast f)(x,y) 
& \ = \  \sum_{z \in [x,y]} f(z,y).
\end{align}
The convolution inverse of $\zeta$ is the Mobius function $\mu$, which satisfies the identity
\be \delta(x,y)\ =\ (\mu \ast \zeta)(x,y)\ =\ \sum_{z \in [x,y]} \mu(x,z)\ =\ \begin{cases} 1 & \text{if}\ x = y \\ 0 & \text{otherwise.}\end{cases}\ee

We will use Mobius inversion on functions from $P$ to $\C$.
The incidence algebra acts on functions (on the left) as follows. For $f \in \ml{A}$ and $g : P \to \C$, we define
\be (f \ast g)(x)\ =\ \sum_{x \preceq y} f(x,y) g(y),\ee
and the Mobius inversion formula is given by
\be
f \ =\ (\zeta \ast g) \qquad \ \ \Leftrightarrow\ \ \qquad g\ =\ (\mu \ast f); \ee or, more explicitly,
\be (\forall x)\ f(x)\ =\ \sum_{x \preceq y} g(y) \ \ \Leftrightarrow\ \ (\forall x)\ g(x)\ =\ \sum_{x \preceq y} \mu(x,y)f(y).
\ee

The Mobius function of $\Pi(n)$ is known (see for example \cite{Rot}): if $\ul{F} \preceq \ul{G}$ and the $i$\textsuperscript{th} block of $\ul{G}$ is a union of $b_i$ blocks of $\ul{F}$, then
\be
\mu(\ul{F},\ul{G})\ =\ (-1)^{\nu(\ul{F}) - \nu(\ul{G})} \prod_{i=1}^{\nu(\ul{G})} (b_i - 1)!.
\ee
The coefficients $\mu(\ul{O},\ul{F})$ and $\mu(\ul{F},\ul{N})$ will often show up in our sums and are given by
\begin{align}
\mu(\ul{O},\ul{F}) &\ =\ (-1)^{n-\nu(\ul{F})} \prod_{i=1}^{\nu(\ul{F})}(|F_i|-1)!, \\
\mu(\ul{F},\ul{N}) &\ =\ (-1)^{\nu(\ul{F})-1} (\nu(\ul{F})-1)!.
\end{align}
%




We make extensive use of the following definition.
\begin{definition}If $\ul{F} \preceq \ul{G} \in \Pi(n)$ are partitions, we say $\ul{F}$ is a \emph{2-refinement} of $\ul{G}$ (or $\ul{G}$ is \emph{2-coarser} than $\ul{F}$) if each block of $\ul{G}$ is a union of at most 2 blocks of $\ul{F}$. If only one block decomposes, we say $\ul{G}$ \emph{covers} $\ul{F}$.
\end{definition}

Covers and 2-refinements arise in our sums, and we note that in these cases the Mobius function simplifies to $\mu(\ul{F},\ul{G}) = (-1)^{\nu(\ul{F}) - \nu(\ul{G})}$. Also, if $\ul{G}$ covers $\ul{F}$ via the decomposition $F_i \cup F_j = G_k$, it's easy to see that
\be \label{cvr_mobius}
\frac{\mu(\ul{O},\ul{F})}{\mu(\ul{O},\ul{G})}\ =\ -\frac{(|F_i| - 1)! (|F_j|-1)!}{(|G_k| - 1)!}.
\ee
More generally, for a 2-refinement $\ul{F} \preceq \ul{G}$, let $\ul{F^l}\in [\ul{F}, \ul{G}]$ be the partition obtained by only decomposing the $l$\textsuperscript{th} block of $G$ into blocks from $\ul{F}$, say $G_l = F_{l_1} \cup F_{l_2}$. Then
\be \label{2refmobius}
\frac{\mu(\ul{O},\ul{F})}{\mu(\ul{O},\ul{G})}\ =\ \prod_{l} \frac{\mu(\ul{O},\ul{F^l})}{\mu(\ul{O},\ul{G})}\ =\ (-1)^{\nu(\ul{F})-\nu(\ul{G})} \prod_l \frac{(|F_{l_1}| - 1)! (|F_{l_2}|-1)!}{(|G_l| - 1)!},
\ee
where $l$ runs over the blocks $G_l$ that decompose in $\ul{F}$.

\begin{definition} \label{s,sc,w,wc} If $\ul{F}$ is a 2-refinement of $\ul{G}$, we define the sets
\begin{align}
S(\ul{F},\ul{G}) &\ = \  \{l : G_l \text{ decomposes in } \ul{F}\}, \nonumber\\
S^c(\ul{F},\ul{G}) &\ = \  \{l : G_l \text{ is a block of } \ul{F}\}, \nonumber\\
W(\ul{F},\ul{G}) &\ = \  \{l : F_l \text{ joins with another block of $\ul{F}$ in } \ul{G}\}, \nonumber\\
W^c(\ul{F},\ul{G}) &\ = \  \{l : F_l \text{ remains a block in } \ul{G}\},
\end{align}
so $S \cup S^c = \{1, \ldots, \nu(\ul{G})\}$ and $W \cup W^c = \{1, \ldots, \nu(\ul{F})\}$.
\end{definition}
\begin{rmk} \label{rmk:pairs_and_2refs}
Given $\ul{G} \in \Pi(n)$, a 2-refinement $\ul{F}$ is uniquely specified by a choice of blocks $S \subseteq \{1, \ldots, \nu(\ul{G})\}$, and, for each $l \in S$, a choice of decomposition $G_l = H_l \cup H^c_l$. (If $|G_l| = 1$ for some $l \in S$, there are no valid decompositions of $G_l$.)

Conversely, given $\ul{F}$, a partition $\ul{G}$ 2-coarser than $\ul{F}$ is uniquely specified by a choice of blocks $W \subseteq \{1, \ldots, \nu(\ul{F})\}$ with $|W|$ even, and a way of pairing up the elements of $W$.
\end{rmk}

\subsection{The combinatorial obstruction} We can now clarify some of the obstacles we need to address.


The first reason the random matrix theory and number theory densities in \cite{Gao} appear different is as follows. In the random matrix theory density, for each partition $\ul{F} = \{F_1, \ldots, F_k\}$ we at one point consider all ways of decomposing some or all of the blocks $F_i$ into exactly two proper nonempty subsets each. That is, we consider all the 2-refinements $\ul{F'}$ of $\ul{F}$. On the number theory side, we instead consider all the ways of pairing up (some or all of) the blocks $F_i$. In other words, we consider all the partitions $\ul{G}$ of which $\ul{F}$ is a 2-refinement. Because the counting is `backwards' here, the terms appear very different from those encountered on the other side. By reindexing these sums appropriately, we are able to match up the parts of the random matrix and number theory densities related to 2-refinements. We then reduce the remaining difference to a Fourier transform identity.

We verify this remaining Fourier transform identity up to the case $n=7$ by breaking down the remaining combinatorics. The difference between our approach and Gao's is as follows. Gao verified the cases $n = 1, 2, 3$ by using various ad hoc Fourier Transform identities, and explicitly computing formulas for (sums of) integrals over certain regions in $\R^n$ ($n \leq 3$), such as (equation 5.11 from \cite{Gao}) :
\begin{align}
&\int_{\substack{ \R^3_{\geq 0} \\ u_1 > 1 + u_2 + u_3}} \prod_{i=1}^3 \hat{f}_i(u_i)du_i\ =\ \int_1^{\infty} \int_0^{u_1-1} \int_0^{u_1-u_2-1} \prod_{i=1}^3 \hat{f}_i(u_i)du_i, \nonumber\\
&\int_{\substack{ \R^3_{\geq 0} \\ u_2 > 1 + u_1 + u_3}} \prod_{i=1}^3 \hat{f}_i(u_i)du_i\ =\ \int_0^{\infty} \int_{1+u_1}^{\infty} \int_0^{u_2-u_1-1} \prod_{i=1}^3 \hat{f}_i(u_i)du_i,
\end{align}
and showed that these sums yielded zero over various sub-regions of the support region $|u_1| + |u_2| + |u_3| < 2$. In contrast, we will write
\begin{align}
\int_{\substack{ \R^3_{\geq 0} \\ u_1 > 1 + u_2 + u_3 }} \prod_{i=1}^3 \hat{f}_i(u_i)du_i &\ = \  \int_{\R^3_{\geq 0}} \tilde{\chi}(u_1 - u_2 - u_3) \prod_{i=1}^3 \hat{f}_i(u_i)du_i, \nonumber\\
\int_{\substack{ \R^3_{\geq 0} \\ u_1 > 1 + u_2 + u_3 }} \prod_{i=1}^3 \hat{f}_i(u_i)du_i &\ = \  \int_{\R^3_{\geq 0}} \tilde{\chi}(- u_1 + u_2 - u_3) \prod_{i=1}^3 \hat{f}_i(u_i)du_i,
\end{align}
where $\tilde{\chi}$ is the indicator function of the interval $[1, \infty)$, and show equality by analyzing the combinatorics of various sums of products of indicator functions.

\section{Recasting the expansions}\label{sec:fixingcounting}

In this section we rewrite both sides to facilitate the comparison, and reduce the problem to a Fourier transform identity. To state the random matrix theory expansion we need the following definition.

\begin{defi}[$\chi^*$]\label{appendix:chistar}
For an integer $k \geq 1$, the sum of indicator functions $\chi^*_k$ on $\R^k$ is defined by
\begin{align} \label{chistar}
\chi^*_k(u_1, \ldots, u_k) &\ = \  \sum_{\substack{ \pi \in S_k \\ \pi(1)=1}} \left(\prod_{i=1}^k \chi(u_{\pi(1)} + \cdots + u_{\pi(i)} - u_{\pi(i+1)} - \cdots - u_{\pi(k)}) \right),
\end{align}
where $S_k$ is the symmetric group on $\{1,\ldots,k\}$, we sum over the $(k-1)!$ permutations fixing 1, and $\chi$ is the indicator function of $[-1,1]$.
\end{defi}
Note that $\chi_n^*$ is symmetric in the variables $u_2, \ldots, u_n$ but not $u_1$. We have the following, however:
\begin{prop}
Fix $m \in \{1, \ldots, n\}$. Let $\phi_{m;n}$ be defined the same way as in \eqref{chistar}, but with the condition ``$\pi(1) = 1$'' replaced by ``$\pi(1)=m$'' (and $k=n$), so
\begin{align} \label{chistar}
\phi_{m;n}(u_1, \ldots, u_n) &\ = \  \sum_{\substack{ \pi \in S_n \\ \pi(1)=m}} \left(\prod_{i=1}^n \chi(u_{\pi(1)} + \cdots + u_{\pi(i)} - u_{\pi(i+1)} - \cdots - u_{\pi(n)}) \right),
\end{align}
Let $f_1, \ldots, f_n$ be even Schwartz functions. Then
\begin{equation} \label{chistar-symmetry}
\int_{\R^n} \phi_{m;n}(u_1, \ldots, u_n) \prod_{i=1}^n f_i(u_i) du_i = \int_{\R^n} \chi^*_n(u_1, \ldots, u_n) \prod_{i=1}^n f_i(u_i) du_i.
\end{equation}
\end{prop}
All our integrals will be against even functions $f_1, \ldots, f_n$, so by abuse of notation, we will sometimes refer to $\chi^*_G$, where $G$ is a set (generally a block of a partition $\ul{G} \in \Pi(n)$). The definition is the same as above (with $k=|G|$) and $G$ is understood the set of indices for the variables $u_i$; any $i \in G$ can play the role of the `distinguished' element 1.
\begin{proof}
We consider the summands of $\chi^*_n$ one at a time. Let $\pi$ be a permutation with $\pi(1)=1$ and let $k$ be such that $\pi(k)=m$. We show that the $\pi$ term gives the same integral as the term in $\phi_{m;n}$ coming from the permutation $\pi'$, where
\begin{equation}
\pi' = \begin{pmatrix} 1 & 2 & \cdots & n-k+1 & n-k & \cdots &n  \\ \pi(k) & \pi(k+1) & \cdots & \pi(n) & \pi(1) & \cdots & \pi(k-1) \end{pmatrix}.
\end{equation}
In particular, $\pi'(1) = \pi(k) = m$ and thus $\pi'$ is one of the terms in $\phi_{m;n}$.

To see this, first replace $u_{\pi(i)} \mapsto -u_{\pi(i)}$ for $i=k, \ldots, n$ in each $\chi$ factor. This doesn't change the value of the integral because the $f_i$ are all even. Now $u_{\pi(k)}$ appears with a negative sign in the $k$\textsuperscript{th} through $n$\textsuperscript{th} factors. Since $\chi$ is an even function, multiply its argument by -1 on each of these factors, so that $u_{\pi(k)}$ now always appears with a positive sign. Reordering the factors cyclically, so that the $k$\textsuperscript{th} term appears first (followed by the $(k+1)$\textsuperscript{st}, $\ldots, n$\textsuperscript{th}, $1$\textsuperscript{st}, $\dots$, $(k-1)$\textsuperscript{th}), gives the desired expression.

This process is invertible, so the terms of $\phi_{m;n}$ are in one-to-one correspondence with the terms of $\chi^*_n$ (and this correspondence preserves the values of the integrals).
\end{proof}

\subsection{Recasting the random matrix side} The $n$-level eigenvalue density for USp (see equation (4.12) in \cite{Gao}) is
\begin{align} \label{gaormt}
\int_{\R^n} \prod_{i=1}^n f_i(x) W^{(n)}_{USp}(x)dx\ =\ \sum_{\underline{G} \in \Pi(n)} (-2)^{n - \nu(\underline{G})} \prod_{l=1}^{\nu(\underline{G})} (P_l + Q_l + R_l),
\end{align}
where
\begin{align}
P_l &\ = \  (|G_l| - 1)! \bigg( (\frac{-1}{2}) \int_{\R} \hat{G}_l(u)du + \int_{\R} G_l(x) dx\bigg), \\
\label{Ql_term}
Q_l &\ = \  - \sum_{[H , H^c]} (|H| - 1)! (|H^c| - 1)! \int_{\R} |u| \widehat{H}(u) \widehat{H^c}(u) du, \\
R_l &\ = \  \frac{1}{2} \int_{\R^{|G_l|}} \bigg((|G_l|-1)!  - \chi^*_{G_l}(u_{i_1},\ldots,u_{i_{|G_l|}}) \bigg) \prod_{i \in G_l} \hat{f}_i(u_i)du_i,
\end{align}
with $\underline{G} = \{G_1, \ldots, G_{\nu(\underline{G})}\}$ and $G_l(x) = \prod_{i \in G_l} f_i(x).$ Also, the sum $\sum_{[H,H^c]}$ ranges over the ways of decomposing $G_l$ into two proper, nonempty disjoint subsets $H$ and $H^c$, and $\displaystyle{\widehat{H}(u) = \widehat{\prod_{i \in H} f_i}(u)}$ and similarly for $H^c$. Except for Lemma \ref{lem:nosupport}, we do not need the expansion of $\chi^*_{G_l}$ until \S\ref{sec:fourier_identity_n6}.

In this section we alter this expression in two ways. First, we rearrange the formula so that the $Q_l$ terms (involving decompositions of the blocks of $\ul{G}$) are put in a form described by 2-refinements of $\ul{G}$. When we work with the number theory side, we perform a similar rearrangement that makes it easy to see the correspondence between these terms. The second improvement is to reduce the number of $R_l$ terms we must analyze by showing that any product of two $R_\ell$ terms vanishes due to support restrictions.

\subsubsection{Reindexing the RMT side} We first recast the above formula in terms of 2-refinements of $\ul{G}$.

\begin{lemma} \label{lem:rmt_with_2refs}
Equation $\eqref{gaormt}$ is equivalent to
\begin{align} \label{rmt_with_2refs}
\int_{\R^n} \prod_{i=1}^n f_i(x) W^{(n)}_{USp}(x)dx\ =\
\sum_{\underline{G} \in \Pi(n)} \sum_{\ul{F} \preceq \ul{G}}^{\op{2ref}} 2^{n - \nu(\underline{G})} \mu(\ul{O},\ul{F}) D(\ul{F},\ul{G})
\prod_{l \in S^c} (A_l + C_l),
\end{align}
where $\sum_{\ul{F}\preceq\ul{G}}^{\op{2ref}}$ runs over all the 2-refinements of $\ul{G}$ (including $\ul{G}$ itself) and
\begin{align}
\label{dfg}
D(\ul{F},\ul{G}) &\ = \  \prod_{l \in S(\ul{F},\ul{G})} \int_\R |u|\widehat{H_l}(u)\widehat{H^c_l}(u)du, \\
\label{A_lone}
A_l &\ = \  -\frac{1}{2} \int_{\R} \widehat{G_l}(u)du + \int_{\R} G_l(x) dx, \\
\label{C_lone}
C_l &\ = \  \frac{1}{2} \int_{\R^{|G_l|}} \bigg(1  - \frac{\chi^*_{G_l}(u_{i_1},\ldots,u_{i_{|G_l|}})}{(|G_l|-1)!}  \bigg) \prod_{i \in G_l} \hat{f}_i(u_i)du_i,
\end{align}
and $H_l \cup H_l^c = G_l$ is the decomposition of $G_l$ in $\ul{F}$, with $\displaystyle{\widehat{H_l}(u) = \widehat{\prod_{i \in H_l} f_i}(u)}$ and similarly for $\widehat{H^c_l}(u)$ (note empty products are 1). The sets $S = S(\ul{F},\ul{G})$ and $S^c = S^c(\ul{F},\ul{G})$ are as in Definition \ref{s,sc,w,wc}.
\end{lemma}

\begin{proof}
We view the sum $\sum_{[H,H^c]}$ in $\eqref{gaormt}$ as a sum $\sum^{\text{cvr},G_l}_{\ul{F} \prec \ul{G}}$ over all strictly finer partitions $\ul{F} \prec \ul{G}$ that are covered by $\ul{G}$ via a decomposition of $G_l$ into $H \cup H^c$. Note that if $G_l$ is a singleton set, then we take the empty sum to be 0. Also, we pull the $(|G_l| - 1)!$ and $(-1)^{n-\nu(\ul{G})}$ factors to the front, to make a $\mu(\ul{O},\ul{F})$ coefficient. From  \eqref{cvr_mobius} we have
\be
-\frac{(|H| - 1)! (|H^c| - 1)!}{(|F_l|-1)!}\ =\ \frac{\mu(\ul{O},\ul{F})}{\mu(\ul{O},\ul{G})}.
\ee

The new RMT formula is then
\begin{align}
\int_{\R^n} \prod_{i=1}^n f_i(x) W^{(n)}_{USp}(x)dx\ =\ \sum_{\underline{G} \in \Pi(n)} 2^{n - \nu(\underline{G})} \mu(\ul{O},\ul{G}) \prod_{l=1}^{\nu(\underline{G})} (A_l + K_l + C_l),
\end{align}
where $A_l$ and $C_l$ are as in equations \eqref{A_lone} and \eqref{C_lone}, and
\begin{align}
K_l &\ = \  \sum^{\text{cvr},G_l}_{\ul{F} \prec \ul{G}} \frac{\mu(\ul{O},\ul{F})}{\mu(\ul{O},\ul{G})} \int_{\R} |u| \widehat{H}(u) \widehat{H^c}(u) du,
\label{kl_term}
\end{align}
where $\displaystyle{\widehat{H}(u) = \widehat{\prod_{f_i \in H}f_i}(u)}$, and similarly for $\widehat{H^c}(u)$.

Now, we begin expanding the product $\prod(A_l + K_l + C_l)$ to work directly with the $K_l$ term. The goal is to re-express these terms as sums over 2-refinements of $\ul{G}$. We have
\begin{align}
\int_{\R^n} \prod_{i=1}^n f_i(x) W^{(n)}_{USp}(x)dx = \sum_{\underline{G} \in \Pi(n)} 2^{n - \nu(\underline{G})} \mu(\ul{O},\ul{G}) \bigg( \sum_{S \subseteq \{1, \ldots, \nu(\ul{G})\}} \prod_{l \in S} K_l \prod_{l \in S^c} (A_l + C_l) \bigg).
\end{align}

We first have the following lemma, which converts the $K_l$ term from a sum over partitions covered by $\ul{G}$ into a sum over 2-refinements of $\ul{G}$.

\begin{lemma} \label{lem:cvr_to_2ref}
Let $\ul{G} \in \Pi(n)$ and let $S \subseteq \{1, \ldots, \nu(\ul{G})\}$ be a fixed subset (i.e., a fixed choice of blocks of $\ul{G}$). Then
\be \label{cvr_to_2ref}
\mu(\ul{O},\ul{G}) \prod_{l \in S} K_l\ =\ \sum_{\ul{F} \preceq \ul{G}}^{\op{2ref},S} \mu(\ul{O},\ul{F}) D(\ul{F},\ul{G}),
\ee
where $\sum_{\ul{F} \preceq \ul{G}}^{\op{2ref},S}$ runs over all the 2-refinements $\ul{F}$ of $\ul{G}$ such that $S(\ul{F},\ul{G}) = S$ is the set of blocks of $\ul{G}$ that decompose in $\ul{F}$. The term $D(\ul{F},\ul{G})$ is as in \eqref{dfg} and $K_l$ is as in \eqref{kl_term}.
\end{lemma}

\begin{rmk}
In order to have any 2-refinements $\ul{F}$ of $\ul{G}$ in the right-hand side of \eqref{cvr_to_2ref} above, each of the blocks $G_l$ ($l\in S$) must not be a singleton set. Note \eqref{cvr_to_2ref} holds either way. If $G_l$ is a singleton set for some $l \in S$, the $K_l$ factor on the left-hand side and the entire right-hand side are both empty sums, hence zero.
\end{rmk}

\begin{proof}[Proof of Lemma \ref{lem:cvr_to_2ref}]
Expanding the left-hand side, we have
\be
\mu(\ul{O},\ul{G}) \prod_{l \in S} K_l\ =\ \mu(\ul{O},\ul{G}) \prod_{l \in S} \sum_{\ul{F} \preceq \ul{G}}^{\op{cvr},G_l} \frac{\mu(\ul{O},\ul{F})}{\mu(\ul{O},\ul{G})} \int_{\R} |u| \widehat{H}(u) \widehat{H^c}(u) du.
\ee
When we expand this sum, we obtain a sum of terms, each of the form
\be
\mu(\ul{O},\ul{G}) \cdot \prod_{l \in S} \frac{\mu(\ul{O},\ul{F^l})}{\mu(\ul{O},\ul{G})} \int_{\R}|u| \widehat{H_l}(u)\widehat{H_l^c}(u)du,
\ee
where $\ul{F^l}$ is the partition covered by $\ul{G}$ by decomposing the block $G_l = H_l \cup H_l^c$ and leaving the other blocks of $\ul{G}$ unchanged.

Let $\ul{F} \preceq \ul{G}$ be the partition obtained by decomposing all the $G_l$ this way. Then each summand corresponds to a unique such $\ul{F}$, a 2-refinement of $\ul{G}$ with $S(\ul{F},\ul{G}) = S$. By the identity $\eqref{2refmobius}$, the $\mu$ coefficient becomes
\be
\mu(\ul{O},\ul{G}) \cdot \prod_{l \in S} \frac{\mu(\ul{O},\ul{F^l})}{\mu(\ul{O},\ul{G})} \ = \  \mu(\ul{O},\ul{G}) \cdot \frac{\mu(\ul{O},\ul{F})}{\mu(\ul{O},\ul{G})} \ = \  \mu(\ul{O},\ul{F}),
\ee
so the term simplifies to
\be
\mu(\ul{O},\ul{F}) \prod_{l \in S} \int_{\R} |u| \widehat{H_l}(u)\widehat{H^c_l}(u)du \ = \  \mu(\ul{O},\ul{F}) D(\ul{F},\ul{G}),
\ee
as desired.

Conversely, every 2-refinement $\ul{F} \preceq \ul{G}$ with $S(\ul{F},\ul{G}) = S$ arises (once) this way, so the two sides of $\eqref{cvr_to_2ref}$ match.
\end{proof}

We now return to the proof of Lemma \ref{lem:rmt_with_2refs}. Next, when we sum $\eqref{cvr_to_2ref}$ over all $S \subseteq \{1, \ldots, \nu(\ul{G})\}$, we get a sum over all the 2-refinements $\ul{F}$ of $\ul{G}$ (including $\ul{G}$ itself, from $S = \eset$). We have
\begin{align}
\int_{\R^n} \prod_{i=1}^n f_i(x) W^{(n)}_{USp}(x)dx \ = \
\sum_{\underline{G} \in \Pi(n)} 2^{n - \nu(\underline{G})} \sum_{\ul{F} \preceq \ul{G}}^{\op{2ref}} \mu(\ul{O},\ul{F}) D(\ul{F},\ul{G})
\prod_{l \in S^c} (A_l + C_l),
\end{align}
where $\sum_{\ul{F}\preceq\ul{G}}^{\op{2ref}}$ runs over all the 2-refinements of $\ul{G}$ (including $\ul{G}$ itself),
completing the proof of Lemma \ref{lem:rmt_with_2refs}.
\end{proof}

\bigskip

\subsubsection{Expanding the $C_l$ terms}

We expand and simplify the $\prod_l (A_l + C_l)$ term. The following lemma drastically reduces the number of terms we have to analyze.

\begin{lemma} \label{lem:nosupport}
Let $G_l$ and $G_k$ be disjoint subsets of $\{1, \ldots, n\}$. Then
\begin{align}
C_l \cdot C_k &\ = \  \nonumber
\int_{\R^{|G_l|}}\bigg( 1-\frac{\chi^*_{G_l}(u)}{(|G_l|-1)!} \bigg)\prod_{i \in G_l}\hat{f}_i(u_i)du_i \cdot \int_{\R^{|G_k|}}\bigg( 1-\frac{\chi^*_{G_k}(u)}{(|G_k|-1)!} \bigg)\prod_{i \in G_k}\hat{f}_i(u_i)du_i \\
&\ = \  0,
\end{align}
where $\chi^*_{G_l}(u)$ is shorthand for $\chi^*_{G_l}(u_{i_1},\ldots, u_{i_{|G_l|}})$, as defined in equation \eqref{chistar}. \end{lemma}

\begin{proof}
Since $G_l$ and $G_k$ are disjoint, we must have either
\be
\sum_{i \in G_l} \supp(\hat{f}_i) < 1\  \ \text{ or }\ \ \sum_{j \in G_k} \supp(\hat{f}_j) < 1,
\ee
since the total support is less than 2. Without loss of generality, assume $G_l$'s total support is less than 1. Then
\be
|\underbrace{\vep_{i_1} u_{i_1} + \cdots + \vep_{i_k} u_{i_k}}_{i_j \in G_l}|\ \leq\ \sum_{G_l} |u_i|\ <\ 1
\ee
in the region of support, so $\chi(\sum_{F_l} \vep_i u_i) = 1$ for any $\vep_i = \pm 1$. Since $\chi^*_{G_l}$ is a sum of $(|G_l|-1)!$ products of $\chi$'s, the $G_l$ integrand is identically 0.
\end{proof}

To emphasize the significance of this lemma, we note that instead of having to expand a product of the form $\prod_{l=1}^k (A_l + C_l)$ into $2^k$ terms
\be
\prod_{l =1}^k(A_l + C_l) \ = \  \sum_{U \subseteq \{1, \ldots, n\}} \prod_{l \in U} A_l \prod_{l \notin U} C_l,
\ee
we only end up with $k+1$ nonvanishing terms:
\be
\prod_{l=1}^k(A_l + C_l) \ = \  \prod_{l=1}^k A_l + \sum_{l=1}^k C_l \cdot \prod_{l' \ne l} A_{l'}.
\ee

Combining Lemmas \ref{lem:nosupport} and \ref{lem:rmt_with_2refs} yields the following.

\begin{lemma} \label{lem:rmt_nosupport_2refs}
With notation as in Lemma \ref{lem:rmt_with_2refs}, $\eqref{gaormt}$ is equivalent to
\begin{align}
\nonumber&\int_{\R^n} \prod_{i=1}^n f_i(x) W^{(n)}_{USp}(x)dx \ = \  \\
&\sum_{\underline{G} \in \Pi(n)} \sum_{\ul{F} \preceq \ul{G}}^{\op{2ref}} 2^{n - \nu(\underline{G})} \mu(\ul{O},\ul{F}) D(\ul{F},\ul{G})
\bigg( \prod_{l \in S^c} A_l + \sum_{l \in S^c} C_l \cdot \prod_{l' \ne l} A_{l'} \bigg).
\label{rmt_nosupport_2refs}
\end{align}
Here $S^c = S^c(\ul{F},\ul{G})$ is as in Definition \ref{s,sc,w,wc}.
\end{lemma}

\noindent The expression $\eqref{rmt_nosupport_2refs}$ is the one we use when we start matching terms with the number theory (NT) side.

\subsection{Recasting the NT formula}
We now recast the NT density as a sum over 2-refinements of partitions, bringing it closer to the RMT formula established in Lemma $\ref{lem:rmt_with_2refs}$. This allows us to fully match one set of terms appearing on both sides. We then alter each formula slightly to reduce the problem to a Fourier transform identity, relating the terms $C_l$ on the RMT side (equation $\eqref{C_lone}$) to the integrals over $\R^k_{\geq 0}$ on the number theory side (equation $\eqref{gaocumulant1}$). 

Gao's expression for the $n$-level density of zeros of quadratic Dirichlet L-functions, which we abbreviate as $W^{(n)}_{Q}$, is (adapted from equation (2.16) in \cite{Gao}):

\begin{align}
\label{gaont2}
\int_{\R^n} \prod_{i=1}^n f_i(x) W^{(n)}_{Q}(x)dx \ = \
 &\lim_{X \to \infty} \frac{\pi^2}{4X} \sum_{d \in D(X)} \sum_{\underline{F} \in \Pi(n)} 2^{n - \nu(\underline{F})} \mu(\ul{O},\ul{F}) \prod_{l = 1}^{\nu(\underline{F})}\big( A_l + B_l \big),
\end{align}
where
\begin{align}
A_l &\ = \  \int_{\R} F_l(x) dx -\frac{1}{2} \int_{\R} \hat{F}_l(u) du, \\
B_l &\ = \  - \frac{2}{\log X} \sumint{F_l}. \label{sumint_series}
\end{align}
Here $d$ is the conductor, $\underline{F} = \{F_1, \ldots, F_{\nu(\underline{F})}\}$ and $F_l(x) = \prod_{i \in F_l} f_i(x)$, $\sum_p$ is over the primes and $\big(\small{\frac{8d}{p}}\big)$ is the Legendre symbol.

Note that the $A_l$ terms are independent of $d$ and $X$. Hence, if we expand the products, the $A_l$ terms can be pulled past  $\lim_{X \to \infty} \sum_{d \in D(X)}$, making their contributions easy to analyze:
\be
\lim_{X \to \infty} \frac{\pi^2}{4X} \sum_{d \in D(X)} \bigg(\prod_l A_l \bigg) \bigg(\prod_{l'} B_{l'} \bigg) \ = \  \bigg( \prod_l A_l \bigg) \cdot \bigg( \lim_{X \to \infty} \frac{\pi^2}{4X} \sum_{d \in D(X)} \prod_{l'} B_{l'} \bigg).
\ee
The main difficulty comes from the expressions \be \lim_{X \to \infty} \frac{\pi^2}{4X} \sum_{d \in D(X)} \prod_{l \in W} B_l, \ee where $W \subseteq \{1, \ldots, \nu(\underline{F})\},$ since the Legendre symbol $\big(\small{\frac{8d}{p}}\big)$ in the series $\eqref{sumint_series}$ introduces a dependence on $d$ and $X$.

For these, Gao develops the following formula (see equation (3.13) in \cite{Gao}):

\begin{lemma}
Let $\ul{F} = \{F_1, \ldots, F_{\nu(\ul{F})}\}$ be as above, and let $W \subseteq \{1, \ldots, \nu(\ul{F})\}$. Then
\begin{align} \label{gaocumulant1}
&\lim_{X \to \infty} \frac{\pi^2}{4X} \sum_{d \in D(X)} \prod_{l\in W} B_l \ = \  \\
\nonumber &\left(\frac{1 + (-1)^{|W|}}{2}\right) 2^{|W|} \sum_{(A;B)} \prod_{i=1}^{|W|/2} \int_0^{\infty} u_i \widehat{F}_{a_i}(u_i)\widehat{F}_{b_i}(u_i)du_i + (-2)^{|W|-1} \sum_{\substack{W_2 \subsetneq W \\ |W_2| \text{ even}}}\bigg(\sum_{(C;D)} \prod_{i=1}^{|W_2|/2}\\
\nonumber & \int_0^{\infty} u_i \widehat{F}_{c_i}(u_i)\widehat{F}_{d_i}(u_i)du_i \bigg)
\cdot
\bigg(\int_{\R^{|W^c_2|}_{\geq 0}} \bigg( \sum_{I \subsetneq W^c_2} (-1)^{|I|} \tilde{\chi}(\sum_{I^c} u_i - \sum_I u_i) \bigg) \prod_{W^c_2} \widehat{F}_i(u_i)du_i \bigg),
\end{align}
where $W^c_2 = W \setminus W_2$, and the notations $\sum_{(A;B)}$ and $\sum_{(C;D)}$ run over the ways of pairing up the elements of $W$ and $W_2$, respectively. Also, $\tilde{\chi}$ is the indicator function of the interval $(1, \infty)$.
Empty products are 1.
\end{lemma}

We now obtain the general formula for the $n$-level density by combining expressions \eqref{gaont2} and \eqref{gaocumulant1} and using the expansion
\be
\prod_{i=1}^{\nu(\underline{F})} (A_l + B_l) \ = \  \sum_{W \subseteq \{1, \ldots, \nu(\underline{F})\}} \prod_{W^c} A_l \prod_W B_l.
\ee

\subsubsection{Reindexing the NT side} We put the NT formula in a form closer to the RMT formula, as a sum indexed by 2-refinements of partitions. We establish the following.

\begin{lemma} \label{lem:nt_with_2refs}
Gao's expression $\eqref{gaont2}$ for the NT density is equivalent to
\begin{align} \label{nt_with_2refs}
\nonumber &\int_{\R^n} \prod_{i=1}^n f_i(x) W^{(n)}_{Q}(x)dx \ = \  \\
&\sum_{\underline{G} \in \Pi(n)} \sum_{\ul{F} \preceq \ul{G}}^{\op{2ref}} 2^{n - \nu(\underline{G})} \mu(\ul{O},\ul{F}) D(\ul{F},\ul{G}) \bigg( \prod_{l \in S^c} A_l - \frac{1}{2}\sum_{T \subseteq S^c} E(\ul{G},T) \prod_{l\in S^c- T} A_l \bigg),
\end{align}
where $\sum_{\ul{F} \preceq \ul{G}}^{\op{2ref}}$ runs over the 2-refinements $\ul{F}$ of $\ul{G}$, the sets $S(\ul{F},\ul{G})$ and $S^c = S^c(\ul{F},\ul{G})$ are as in Definition \ref{s,sc,w,wc}, and
\begin{align}
D(\ul{F},\ul{G}) &\ = \  \prod_{l \in S(\ul{F},\ul{G})} \int_\R |u|\widehat{H_l}(u)\widehat{H^c_l}(u)du, \\
A_l &\ = \  \frac{-1}{2} \int_{\R} \widehat{G_l}(u)du + \int_{\R} G_l(x) dx, \\
E(\ul{G},T) &\ = \  2^{|T|} \int_{\R_{\geq0}^{|T|}} \bigg( \sum_{I \subseteq T} (-1)^{|I|} \tilde\chi(\sum_I u_i - \sum_{I^c} u_i ) \bigg) \prod_{l \in T} \widehat{G_l}(u_l)du_l, \label{nt_etermone}
\end{align}
where for $l \in S(\ul{F},\ul{G})$, $G_l = H_l \cup H_l^c$ is the decomposition of the block $G_l$ into blocks of $\ul{F}$, and $\tilde\chi$ is the indicator function of $(1, \infty)$. Empty products are 1 and empty sums, in particular $E(\ul{G},\eset)$, are 0.
\end{lemma}

In order to prove Lemma \ref{lem:nt_with_2refs}, we first alter formula $\eqref{gaocumulant1}$ for the $\prod B_l$ terms.


\begin{lemma} \label{lem:jake_cumulant}
Let $\ul{F} \in \Pi(n)$ and $W \subseteq \{1, \ldots, \nu(\ul{F})\}$. The following formula is equivalent to Gao's Lemma \ref{gaocumulant1}:
\begin{align} \label{jake_cumulant}
\lim_{X \to \infty} \frac{\pi^2}{4X} \sum_{d \in D(X)} \prod_{l\in W} B_l
&\ = \  \left(\frac{1 + (-1)^{|W|}}{2}\right) \bigg( \sum_{\ul{G} \succeq \ul{F}}^{\op{2cor},W} 2^{\nu(\ul{F})-\nu(\ul{G})}  D(\ul{F},\ul{G}) \bigg) \nonumber\\
& - \frac{1}{2} \sum_{\substack{W_2 \subseteq W \\ |W_2| \text{ even}}}
\bigg( \sum_{\ul{G} \succeq \ul{F}}^{\op{2cor},W_2} 2^{\nu(\ul{F})-\nu(\ul{G})}  D(\ul{F},\ul{G}) \bigg) E(\ul{F},W^c_2),
\end{align}
where $\sum_{\ul{G} \succeq \ul{F}}^{\op{2cor},W}$ and $\sum_{\ul{G} \succeq \ul{F}}^{\op{2cor},W_2}$ run over the partitions $\ul{G}$ 2-coarser than $\ul{F}$ with $W(\ul{F},\ul{G}) = W$ and $W_2$, respectively, and the other notation is as in Lemma \ref{lem:nt_with_2refs}.
\end{lemma}

\begin{proof}
If $|W|$ is even, we claim
\be \label{3.18}
2^{|W|} \sum_{(A;B)} \prod_{i=1}^{|W|/2} \int_0^{\infty} u\ \widehat{F}_{a_i}(u)\widehat{F}_{b_i}(u)du
\ = \  \sum_{\ul{G} \succeq \ul{F}}^{\op{2cor},W} 2^{\nu(\ul{F}) - \nu(\ul{G})} D(\ul{F},\ul{G}),
\ee
where the left-hand side is as in Gao's formula, \eqref{gaocumulant1}. (The same identity holds with $W$ replaced by $W_2$.)
%
%
To see this, first note that by Remark $\ref{rmk:pairs_and_2refs}$, each way of pairing up the elements of an even subset $W \subseteq \{1, \ldots, \nu(\ul{F})\}$ corresponds to a unique partition $\ul{G}$ that is 2-coarser than $\ul{F}$, with $W(\ul{F},\ul{G}) = W$. This correspondence is one of the key ingredients, as it allows us to begin expressing the sum in terms of 2-coarser partitions. Later we will switch orders of summation, converting sums over 2-coarser partitions into sums over 2-refinements, which is what we have on the RMT side.

Thus
\be
2^{|W|} \sum_{(A;B)} \prod_{i=1}^{|W|/2} \int_0^{\infty} u\ \widehat{F}_{a_i}(u)\widehat{F}_{b_i}(u)du
\ = \  2^{|W|} \sum_{\ul{G} \succeq \ul{F}}^{\op{2cor},W} \prod_{i=1}^{|W|/2} \int_0^{\infty} u\ \widehat{F}_{a_i}(u)\widehat{F}_{b_i}(u)du.
\ee

For the integrands, observe that pairing $F_{l_1}$ with $F_{l_2}$ to form a block $G_l$ of $\ul{G}$ is equivalent to decomposing $G_l$ into two subsets $G_l = H_l \cup H_l^c$ with $H_l = F_{l_1}$ and $H_l^c = F_{l_2}$. Since each $\widehat{F_{l_i}}$ is an even function, we can replace $\int_0^\infty$ with $\frac{1}{2} \int_\R$, and $u\ \widehat{F_{l_1}}(u) \widehat{F_{l_2}}(u)$ becomes $|u|\widehat{F_{l_1}}(u) \widehat{F_{l_2}}(u)$, as in the definition of the term $D(\ul{F},\ul{G})$:
\be
2^{|W|} \sum_{\ul{G} \succeq \ul{F}}^{\op{2cor},W} \prod_{i=1}^{|W|/2} \int_0^{\infty} u\ \widehat{F}_{a_i}(u)\widehat{F}_{b_i}(u)du
\ = \  2^{|W|} \sum_{\ul{G} \succeq \ul{F}}^{\op{2cor},W} 2^{-|W|/2} D(\ul{F},\ul{G}).
\ee
Finally, for the $2^{|W|/2}$ coefficient, observe that $\nu(\ul{F})-\nu(\ul{G}) = |W|/2$, since each pairing reduces the total number of blocks by 1. We apply identity \eqref{3.18} to both the $\sum_{(A;B)}$ and $\sum_{(C;D)}$ terms in Gao's expression \eqref{gaocumulant1} to obtain the desired form \eqref{jake_cumulant}.
\end{proof}

\noindent We return to the proof of Lemma \ref{lem:nt_with_2refs}. Applying Lemma \ref{lem:jake_cumulant} to Gao's expression \eqref{gaont2} for the NT density gives
\begin{align}
\nonumber&\int_{\R^n} \prod_{i=1}^n f_i(x) W^{(n)}_{Q}(x)dx  \\
%
\nonumber&\ = \  \sum_{\underline{F} \in \Pi(n)} 2^{n - \nu(\underline{F})} \mu(\ul{O},\ul{F}) \sum_{W \subseteq \{1, \ldots, \nu(\ul{F})\}} \Big( \prod_{l\notin W} A_l \Big)  \Big( \lim_{X \to \infty} \frac{\pi^2}{4X} \sum_{d \in D(X)} \prod_{l \in W} B_l \Big) \\
&\ = \  S_1 + S_2,
\end{align}
where
\begin{align}
\label{s1_term}
S_1 &\ = \  \sum_{\underline{F} \in \Pi(n)} 2^{n - \nu(\underline{F})} \mu(\ul{O},\ul{F}) \sum_{\substack{W \subseteq \{1, \ldots, \nu(\ul{F})\} \\ |W| \text{ even}}} \Big( \prod_{l\notin W} A_l \Big) \bigg( \sum_{\ul{G} \succeq \ul{F}}^{\op{2cor},W} 2^{\nu(\ul{F})-\nu(\ul{G})}  D(\ul{F},\ul{G}) \bigg), \\
\label{s2_term}
S_2 &\ = \  - \frac{1}{2} \sum_{\underline{F} \in \Pi(n)} 2^{n - \nu(\underline{F})} \mu(\ul{O},\ul{F}) \sum_{W \subseteq \{1, \ldots, \nu(\ul{F})\}} \Big( \prod_{l\notin W} A_l \Big) \cdot \nonumber\\
& \qquad \qquad \sum_{\substack{W_2 \subseteq W \\ |W_2| \text{ even}}}
\bigg( \sum_{\ul{G} \succeq \ul{F}}^{\op{2cor},W_2} 2^{\nu(\ul{F})-\nu(\ul{G})}  D(\ul{F},\ul{G}) \bigg) E(\ul{F},W^c_2).
 \end{align}

We work with $S_1$ and $S_2$ separately, since $S_2$ includes an extra summation. For $S_1$, we have
\begin{align}
\nonumber
S_1 &\ = \  \sum_{\underline{F} \in \Pi(n)} 2^{n - \nu(\underline{F})} \mu(\ul{O},\ul{F}) \sum_{\substack{W \subseteq \{1, \ldots, \nu(\ul{F})\} \\ |W| \text{ even}}} \Big( \prod_{l\notin W} A_l \Big) \bigg( \sum_{\ul{G} \succeq \ul{F}}^{\op{2cor},W} 2^{\nu(\ul{F})-\nu(\ul{G})}  D(\ul{F},\ul{G}) \bigg) \\
&\ = \  \sum_{\underline{F} \in \Pi(n)} \sum_{\substack{W \subseteq \{1, \ldots, \nu(\ul{F})\} \\ |W| \text{ even}}} \sum_{\ul{G} \succeq \ul{F}}^{\op{2cor},W} 2^{n - \nu(\underline{G})} \mu(\ul{O},\ul{F}) D(\ul{F},\ul{G}) \prod_{l\notin W} A_l.
\end{align}
Now the double sum $\displaystyle{\sum_{\substack{W \subseteq \{1, \ldots, \nu(\ul{F})\} \\ |W| \text{ even}}} \sum_{\ul{G} \succeq \ul{F}}^{\op{2cor},W}}$ is equivalent to summing over all the partitions $\ul{G}$ that are 2-coarser than $\ul{F}$, i.e., $\sum_{\ul{G} \succeq \ul{F}}^{\op{2cor}}$, since every such $\ul{G}$ arises exactly once this way (including $\ul{G} = \ul{F}$, from the case $W = \eset$). Note that the set $\{l \notin W\}$ is just the list of blocks that are common to both $\ul{F}$ and $\ul{G}$, so it is the same as $S^c(\ul{F},\ul{G})$. By switching the order of summation of $\ul{F}$ and $\ul{G}$, we obtain
\begin{align} \label{s1term-final}
 \nonumber S_1 &\ = \  \sum_{\underline{F} \in \Pi(n)} \sum_{\ul{G} \succeq \ul{F}}^{\op{2cor}} 2^{n - \nu(\underline{G})} \mu(\ul{O},\ul{F}) D(\ul{F},\ul{G}) \prod_{l \in S^c(\ul{F},\ul{G})} A_l \\
&\ = \  \sum_{\underline{G} \in \Pi(n)} \sum_{\ul{F} \preceq \ul{G}}^{\op{2ref}} 2^{n - \nu(\underline{G})} \mu(\ul{O},\ul{F}) D(\ul{F},\ul{G}) \prod_{l \in S^c(\ul{F},\ul{G})} A_l.
\end{align}

The term $S_2$ is more delicate, since it includes an extra summation:
\begin{align} \nonumber
S_2 &\ = \  - \frac{1}{2} \sum_{\underline{F} \in \Pi(n)} 2^{n - \nu(\underline{F})} \mu(\ul{O},\ul{F}) \sum_{W \subseteq \{1, \ldots, \nu(\ul{F})\}} \Big( \prod_{l\notin W} A_l \Big) \nonumber\\ & \ \ \ \ \ \ \ \ \ \ \ \ \ \ \ \ \ \ \ \ \  \sum_{\substack{W_2 \subseteq W \\ |W_2| \text{ even}}}
\bigg( \sum_{\ul{G} \succeq \ul{F}}^{\op{2cor},W_2} 2^{\nu(\ul{F})-\nu(\ul{G})}  D(\ul{F},\ul{G}) \bigg) E(\ul{F},W^c_2) \nonumber\\
&\ = \  - \frac{1}{2} \sum_{\underline{F} \in \Pi(n)} \sum_{W \subseteq \{1, \ldots, \nu(\ul{F})\}} \sum_{\substack{W_2 \subseteq W \\ |W_2| \text{ even}}} \sum_{\ul{G} \succeq \ul{F}}^{\op{2cor},W_2} 2^{n - \nu(\underline{G})} \mu(\ul{O},\ul{F})
D(\ul{F},\ul{G}) E(\ul{F},W \setminus W_2) \prod_{l\notin W} A_l.
\end{align}

We switch the choice of subsets $W, W_2 \subseteq \{1, \ldots, \nu(\ul{F})\}$. In particular, the choices of $W$ and $W_2$ effectively partition $\{1, \ldots, \nu(\ul{F})\}$ into three disjoint subsets:\\

\begin{tabular}{rl}
$W_2$ (with $|W_2|$ even): & lists the blocks of $\ul{F}$ merged to form blocks of $\ul{G}$ in the $D(\ul{F},\ul{G})$ term \\
$W - W_2$: & lists the blocks to go in the $E$ term, \\
$\{1, \ldots, \nu(\ul{F})\} - W$: & lists the blocks to go in the $\prod A_l$ term.
\end{tabular} \\

We switch this so that $W_2$ is chosen first, which allows us to pull the $\sum_{\ul{F} \leq \ul{G}}^{\op{2cor},W_2}$ to the front. In other words, we choose $W_2$, followed by a disjoint set $T \subseteq \{1, \ldots, \nu(\ul{F})\} - W_2$, which lists the blocks to go in the $E$ term. With this change, the $E(\ul{F}, W \setminus W_2)$ is replaced by $E(\ul{F}, T)$, and $\prod_{l \notin W} A_l$ becomes $\prod_{l \notin W_2 \cup T} A_l$.
%
Now we can pull the $\sum_{\ul{G} \succeq \ul{F}}^{\op{2cor},W_2}$ outward:
\begin{align} \nonumber
S_2 &\ = \  - \frac{1}{2} \sum_{\underline{F} \in \Pi(n)} \sum_{\substack{W_2 \subseteq \{1, \ldots, \nu(\ul{F})\} \\ |W_2| \text{ even}}} \sum_{T \subseteq W_2^c} \sum_{\ul{G} \succeq \ul{F}}^{\op{2cor},W_2} 2^{n - \nu(\underline{G})} \mu(\ul{O},\ul{F})
D(\ul{F},\ul{G}) E(\ul{F},T) \prod_{l\notin T \cup W_2} A_l \\
&\ = \  - \frac{1}{2} \sum_{\underline{F} \in \Pi(n)} \sum_{\substack{W_2 \subseteq \{1, \ldots, \nu(\ul{F})\} \\ |W_2| \text{ even}}} \sum_{\ul{G} \succeq \ul{F}}^{\op{2cor},W_2}  \sum_{T \subseteq W_2^c} 2^{n - \nu(\underline{G})} \mu(\ul{O},\ul{F})
D(\ul{F},\ul{G}) E(\ul{F},T) \prod_{l\notin T \cup W_2} A_l.
\end{align}
Now the summation $\sum_{|W_2| \text{ even}} \sum_{\ul{G}}^{\op{2cor}, W_2}$ is the same as what we encountered in our analysis of $S_1$, a sum over \emph{all} partitions $\ul{G}$ that are 2-coarser than $\ul{F}$ (including $\ul{G} = \ul{F}$, from the case $W_2 = \eset$). The set $W_2^c$ is the same as $S^c(\ul{F},\ul{G})$, the list of blocks common to both partitions, so we rewrite $\sum_{T \subseteq W_2^c}$ as $\sum_{T \subseteq S^c(\ul{F},\ul{G})}$ and $\prod_{l \notin T\cup W_2}$ with $\prod_{l \in S^c - T}$. We obtain
\begin{align}
S_2 &\ = \  - \frac{1}{2} \sum_{\underline{F} \in \Pi(n)} \sum_{\ul{G} \succeq \ul{F}}^{\op{2cor}} 2^{n - \nu(\underline{G})} \mu(\ul{O},\ul{F}) D(\ul{F},\ul{G})
\sum_{T \subseteq S^c(\ul{F},\ul{G})} E(\ul{F},T) \prod_{l\in S^c- T} A_l.
\end{align}
Now we switch the $\sum_{\ul{F}}$ and $\sum_{\ul{G}}$, converting $S_2$ into a sum over 2-refinements:
\begin{align} \label{s2term-final}
S_2 &\ = \  - \frac{1}{2} \sum_{\underline{G} \in \Pi(n)} \sum_{\ul{F} \preceq \ul{G}}^{\op{2ref}} 2^{n - \nu(\underline{G})} \mu(\ul{O},\ul{F})
D(\ul{F},\ul{G})  \sum_{T \subseteq S^c(\ul{F},\ul{G})} E(\ul{F},T) \prod_{l\in S^c- T} A_l.
\end{align}

Finally, we rewrite $E(\ul{F}, T) = E(\ul{G},T)$. This is just a relabeling, since $T \subseteq S^c,$ the set of blocks $F_l \in \ul{F}$ that are unchanged in $\ul{G}$, and the integral over $\R_{\geq 0}^{|T|}$ in the definition of $E$ (equation \eqref{nt_etermone}) only involves the functions $\widehat{F_l}(u_l)$ where $l \in T$. Nonetheless, it is important as it expresses the $E$ term in terms of the outermost summation $\sum_{\ul{G}}$.

Putting together our expressions $\eqref{s1term-final}$ for $S_1$ and $\eqref{s2term-final}$ for $S_2$ yields an NT formula expressed in terms of 2-refinements:
\begin{align}
\nonumber &\int_{\R^n} \prod_{i=1}^n f_i(x) W^{(n)}_{Q}(x)dx \ = \  \\
&\sum_{\underline{G} \in \Pi(n)} \sum_{\ul{F} \preceq \ul{G}}^{\op{2ref}} 2^{n - \nu(\underline{G})} \mu(\ul{O},\ul{F}) D(\ul{F},\ul{G}) \bigg( \prod_{l \in S^c} A_l - \frac{1}{2}\sum_{T \subseteq S^c} E(\ul{G},T) \prod_{l\in S^c- T} A_l \bigg).
\end{align}
This completes the proof of Lemma \ref{lem:nt_with_2refs}. \qed

\begin{rmk} The key step in the proof of Lemma \ref{lem:nt_with_2refs} was to switch the orders of summation of $\ul{F}$ and $\ul{G}$ in the NT density, replacing a sum over 2-coarser partitions by a sum over 2-refinements. On the RMT side, 2-refinements already appeared naturally as products of the terms $Q_l$ (equation \eqref{Ql_term}), and thus no switch was necessary.
\end{rmk}

\subsection{Reducing to the Fourier identity}

Lemmas \ref{lem:rmt_nosupport_2refs} and \ref{lem:nt_with_2refs} establish the following forms for the RMT and NT density expressions, $W_{\op{USp}}^{(n)}$ and $W_Q^{(n)}$:\\
\begin{align}
&\op{RMT}: \ \ \nonumber\int_{\R^n} \prod_{i=1}^n f_i(x) W^{(n)}_{USp}(x)dx \ = \  \\
\label{3.28}
&\sum_{\underline{G} \in \Pi(n)} \sum_{\ul{F} \preceq \ul{G}}^{\op{2ref}} 2^{n - \nu(\underline{G})} \mu(\ul{O},\ul{F}) D(\ul{F},\ul{G})
\bigg( \prod_{l \in S^c} A_l + \sum_{l \in S^c} C_l \cdot \prod_{l' \ne l} A_{l'} \bigg), \\
&\op{NT}: \ \ \nonumber \int_{\R^n} \prod_{i=1}^n f_i(x) W^{(n)}_{Q}(x)dx \ = \  \\
\label{3.29}
&\sum_{\underline{G} \in \Pi(n)} \sum_{\ul{F} \preceq \ul{G}}^{\op{2ref}} 2^{n - \nu(\underline{G})} \mu(\ul{O},\ul{F}) D(\ul{F},\ul{G}) \bigg( \prod_{l \in S^c} A_l - \frac{1}{2}\sum_{T \subseteq S^c} E(\ul{G},T) \prod_{l\in S^c- T} A_l \bigg),
\end{align}
where $\sum_{\ul{F} \preceq \ul{G}}^{\op{2ref}}$ runs over the 2-refinements $\ul{F}$ of $\ul{G}$, and
\begin{align}
\label{dfg-nt}
D(\ul{F},\ul{G}) &\ = \  \prod_{l \in S(\ul{F},\ul{G})} \int_\R |u|\widehat{H_l}(u)\widehat{H^c_l}(u)du, \\
\label{A_ltwo}
A_l &\ = \  \frac{-1}{2} \int_{\R} \widehat{G_l}(u)du + \int_{\R} G_l(x) dx, \\
\label{C_ltwo}
C_l &\ = \  \frac{1}{2} \int_{\R^{|G_l|}} \bigg(1  - \frac{\chi^*_{G_l}(u_{i_1},\ldots,u_{i_{|G_l|}})}{(|G_l|-1)!}  \bigg) \prod_{i \in G_l} \hat{f}_i(u_i)du_i, \\
E(\ul{G},T) &\ = \  2^{|T|} \int_{\R_{\geq0}^{|T|}} \bigg( \sum_{I \subseteq T} (-1)^{|I|} \tilde\chi(\sum_I u_i - \sum_{I^c} u_i ) \bigg) \prod_{l \in T} \widehat{G_l}(u_l)du_l,
\label{nt_etermtwo}
\end{align}
where for $l \in S(\ul{F},\ul{G})$, $G_l = H_l \cup H_l^c$ is the decomposition of the block $G_l$ into blocks of $\ul{F}$, and $\tilde\chi$ is the indicator function of $(1, \infty)$. Empty products are 1 and empty sums, in particular $E(\ul{G},\eset)$, are 0.

We have some cancelation right away: namely, the terms $\prod_{l \in S^c} A_l$ without $C_l$ or $E(\ul{G},T)$ factors match, since the new expressions count all the $D(\ul{F},\ul{G})$ factors the same way on both sides. Compare this with the original density expressions \eqref{gaormt} and \eqref{gaont2}, which only make it easy to see equality between the terms with all $A_l$ factors (without any of the $C_l$, $E(\ul{G},T)$ or $D(\ul{F},\ul{G})$ factors). Those terms show up in the new expressions as the trivial 2-refinements where $\ul{F} = \ul{G}$.

Unfortunately, with the expressions above, the sums \emph{do not} match term-by-term: the $E$ terms combine across many different 2-refinement pairs $(\ul{G},\ul{F})$. The goal of this section is to reduce the Density Conjecture to an identity relating the $E(\ul{G},T)$ to the $C_l$ terms. We use Mobius inversion to express the identity in a fairly simple way. We then verify the identity for $n \leq 7$ by breaking down the remaining combinatorics.

\subsubsection{Isolating the $C_l$ and $E$ terms}

By canceling the matching terms in the two densities, we are reduced to showing equality between
\begin{align} \label{RMT_fourier}
RMT &\ = \  \sum_{\underline{G} \in \Pi(n)} \sum_{\ul{F} \preceq \ul{G}}^{\op{2ref}} 2^{n - \nu(\underline{G})} \mu(\ul{O},\ul{F}) D(\ul{F},\ul{G})
\sum_{l \in S^c} C_l \cdot \prod_{l' \ne l} A_{l'}, \\
NT &\ = \  - \frac{1}{2}\ \sum_{\underline{G} \in \Pi(n)} \sum_{\ul{F} \preceq \ul{G}}^{\op{2ref}} 2^{n - \nu(\underline{G})} \mu(\ul{O},\ul{F}) D(\ul{F},\ul{G})\sum_{T \subseteq S^c} E(\ul{G},T) \prod_{l\in S^c- T} A_l, \label{NT_fourier}
\end{align}
with notation as in \eqref{dfg-nt}-\eqref{nt_etermtwo}. Note that these expression are {\bf not} the same as the $n$-level density expressions, \eqref{3.28} and \eqref{3.29}: all the matching terms have been removed.

In this section, we rewrite the $C$ and $E$ terms to depend only on $T$, not $\ul{G}$. This allows us to pull them outside the summation $\sum_{\ul{G}} \sum_{\ul{F}}$. We show the following.

\begin{lemma} \label{lem:fourier_forms}
Equations \eqref{RMT_fourier} and \eqref{NT_fourier} are equivalent to
\begin{align}
\label{new_RMT_fourier}
RMT &\ = \  \frac{1}{2} \sum_{\ml{U} \subseteq \{1, \ldots, n\}} \bigg( 2^{|\ml{U}|} C(\ul{O}_{\ml{U}}) \bigg) \cdot \op{Rest}(\ml{U}^c), \\
\label{new_NT_fourier}
NT &\ = \
 \frac{1}{2}\ \sum_{\ml{U} \subseteq \{1, \ldots, n\}} \bigg( 2^{|\ml{U}|}  \sum_{\ml{T} \in \Pi(\ml{U})} \mu(\ul{O}_{\ml{U}},\ul{\ml{T}}) E(\ul{\ml{T}}) \bigg) \cdot \op{Rest}(\ml{U}^c),
\end{align}
where
\begin{align} \label{rest_term}
\op{Rest}(\ml{U}^c) &\ = \  \sum_{\underline{G} \in \Pi(\ml{U}^c)} \sum_{\ul{F} \preceq \ul{G}}^{\op{2ref}} 2^{|\ml{U}^c| - \nu(\underline{G})} \mu(\ul{O}_{\ml{U}^c},\ul{F}) D(\ul{F},\ul{G}) \prod_{l\in S^c} A_l, \\
\label{new_cterm}
C(\ul{\ml{T}}) &\ = \  \frac{1}{2} \int_{\R^{\nu(\ul{\ml{T}})}} \bigg(\mu(\ul{\ml{T}},\ul{N}) + (-1)^{\nu(\ul{\ml{T}})} \chi^*_{\nu(\ul{\ml{T}})}(u_1,\ldots,u_{\nu(\ul{\ml{T}})})  \bigg) \prod_{l = 1}^{\nu(\ul{\ml{T}})} \widehat{\ml{T}_l}(u_l)du_l, \\
\label{new_eterm}
E(\ul{\ml{T}}) &\ = \  \int_{\R_{\geq0}^{\nu(\ul{\ml{T}})}} \bigg( \sum_{I \subseteq \{1, \ldots, \nu(\ul{\ml{T}})\}} (-1)^{|I|+1} \tilde\chi(\sum_I u_i - \sum_{I^c} u_i ) \bigg) \prod_{l=1}^{\nu(\ul{\ml{T}})} \widehat{
\ml{T}_l}(u_l)du_l,
\end{align}
and $\ul{O}_{\ml{U}}$ is the minimal element of $\Pi(\ml{U})$ (all singleton blocks), and the rest of the notation is as above. Note that $\mu(\ul{\ml{T}},\ul{N}) = (-1)^{\nu(\ul{\ml{T}})-1}(\nu(\ul{\ml{T}})-1)!$ and so $\mu(\ul{O}_{\ml{U}},\ul{N}_{\ml{U}}) = (-1)^{|\ml{U}| - 1}(|\ml{U}|-1)!.$
\end{lemma}

\begin{proof}
We first work with the NT side. First of all, from the definition in \eqref{nt_etermtwo},  $E(\ul{G},T)$ is an integral involving only the functions \be \widehat{G_l}(u)\ =\ \widehat{\prod_{i \in G_l} f_i}(u)\ee from the blocks $G_l$, $l \in T$. We obtain the $E$ term by choosing a partition $\ul{G} \in \Pi(n)$, followed by a 2-refinement $\ul{F}$, followed by a choice of blocks $T \subseteq S^c(\ul{F},\ul{G})$.

To isolate the $E$ term, we switch orders. We first choose a subset $\ml{U} \subseteq \{1, \ldots n\}$ of test functions and a partition $\ul{\ml{T}} \in \Pi(\ml{U})$, and then choose a partition of the remaining elements, $\ul{G}' \in \Pi(\ml{U}^c)$, and a 2-refinement $\ul{F}'$ of $\ul{G}'$. We use $E(\ul{\ml{T}})$, as  defined in \eqref{new_eterm}. Note that $E(\ul{G},T) = (-1)\cdot 2^{\nu(\ul{\ml{T}})} E(\ul{\ml{T}})$. (Compare \eqref{new_eterm} and \eqref{nt_etermtwo}.)

So, for each term we have
\be \label{3.37}
E(\ul{G},T) D(\ul{F},\ul{G}) \prod_{l \in S^c(\ul{F},\ul{G}) - T} A_l  \ = \  - 2^{\nu(\ul{\ml{T}})} E(\ul{\ml{T}}) D(\ul{F}',\ul{G}')\prod_{l \in S^c} A_l.
\ee

For the $2^{n-\nu(\ul{G})} \mu(\ul{O},\ul{F})$ coefficient, we have to pull out a factor of $2^{|\ml{U}|-\nu(\ul{\ml{T}})} \mu(\ul{O}_{\ml{U}},\ul{\ml{T}})$ (where $\ul{O}_{\ml{U}}$ is the minimal element of $\Pi(\ml{U})$):
\begin{align} \nonumber
2^{n-\nu(\ul{G})}\mu(\ul{O},\ul{F}) &\ = \  2^{n-\nu(\ul{G})} (-1)^{n-\nu(\ul{F})} \prod_{l=1}^{\nu(\ul{F})}(|F_l|-1)! \\
&\ = \ \bigg(2^{|\ml{U}|-\nu(\ul{\ml{T}})} \mu(\ul{O}_{\ml{U}},\ul{\ml{T}}) \bigg)
\bigg(2^{|\ml{U}^c|-\nu(\ul{G}')} \mu(\ul{O}_{\ml{U}^c},\ul{F}') \bigg).
\end{align}
Note that the $2^{-\nu(\ul{\ml{T}})}$ will cancel with the $2^{\nu(\ul{\ml{T}})}$ coefficient on $E(\ul{\ml{T}})$ in \eqref{3.37}.

This gives the desired expression for the number theory side:
\be
\frac{1}{2}\ \sum_{\ml{U} \subseteq \{1, \ldots, n\}} \bigg( \sum_{\ml{T} \in \Pi(\ml{U})} 2^{|\ml{U}|} \mu(\ul{O},\ul{\ml{T}}) E(\ul{\ml{T}}) \bigg)
\sum_{\underline{G} \in \Pi(\ml{U}^c)} \sum_{\ul{F} \preceq \ul{G}}^{\op{2ref}} 2^{|\ml{U}^c| - \nu(\underline{G})} \mu(\ul{O},\ul{F}) D(\ul{F},\ul{G}) \prod_{l\in S^c} A_l,
\ee
as desired.

We proceed similarly for the random matrix theory side.
First of all, the $C_l$ term appearing in \eqref{C_ltwo} always has the test functions arranged as $\prod_{i \in G_l} \widehat{f_i}(u_i)du_i$, for some block $G_l$ of $\ul{G}$, with each test function Fourier-transformed separately. Thus with notation as in \eqref{new_cterm}, the $C_l$ term always takes the form
\be C_l\ =\ \frac{(-1)^{|\ml{U}|-1}}{(|\ml{U}|-1)!} C(\ul{O}_{\ml{U}})\ =\ \frac{1}{\mu(\ul{O}_{\ml{U}},\ul{N})} C(\ul{O}_{\ml{U}}),\ee where $\ml{U} = G_l \subseteq \{1, \ldots, n\}$ and $\ul{O}_{\ml{U}}$ is the minimal element of $\Pi(\ml{U})$, with each test function in its own (singleton) block. (It is nonetheless necessary to define $C(\ul{\ml{T}})$ for any partition $\ul{\ml{T}} \in \Pi(\ml{U})$, in order to use Mobius inversion later.)

The argument is now similar to (in fact more straightforward than) the NT side. The $C_l$ term in \eqref{NT_fourier} arises choosing a partition $\ul{G} \in \Pi(n)$, a 2-refinement $\ul{F}$ of $\ul{G}$, and a single block $G_l \in S^c(\ul{F},\ul{G})$ to put in the $C_l$ term.

We switch orders. We first choose a subset $\ml{U} \subseteq \{1, \ldots, n\}$ (from which to get a $C(\ul{O}_{\ml{U}})$ term), then choose a partition $\ul{G}' \in \Pi(\ml{U}^c)$ and a 2-refinement $\ul{F}'$ of $\ul{G}'$. As with the RMT side, the $D(\ul{F},\ul{G})$ and $A_l$ terms are unaffected, but we have to break up the Mobius coefficient $\mu(\ul{O},\ul{F})$. As the $(|F_l|-1)!$ factor cancels the $\frac{1}{(|\ml{U}|-1)!}$ factor on the $C(\ul{O}_{\ml{U}})$ term, the coefficient becomes
\begin{align} \nonumber
2^{n-\nu(\ul{G})} \mu(\ul{O},\ul{F}) \cdot C_l &\ = \  2^{n-\nu(\ul{G})} (-1)^{n-\nu(\ul{F})} \prod_{l=1}^{\nu(\ul{F})} (|F_l|-1)! \cdot \frac{(-1)^{|\ml{U}|-1}}{(|\ml{U}|-1)!} C(\ul{O}_{\ml{U}}) \\
&\ = \  \bigg(2^{|\ml{U}|-1} C(\ul{O}_{\ml{U}}) \bigg)  \bigg( 2^{|\ml{U}^c| - \nu(\ul{G}')} \mu(\ul{O}_{\ml{U}^c},\ul{F}') \bigg).
\end{align}

So for a single term, we have
\be
2^{n-\nu(\ul{G})} \mu(\ul{O},\ul{F}) D(\ul{F},\ul{G})\ C_l \ \prod_{l' \ne l} A_{l'} \ = \
2^{|\ml{U}|-1} C(\ul{O}_{\ml{U}}) \bigg(2^{|\ml{U}^c|-\nu(\ul{F}')} \mu(\ul{O}_{\ml{U}^c},\ul{F}') D(\ul{F}',\ul{G}') \prod_{l \in S^c} A_l\bigg).
\ee

Thus the RMT side becomes
\be
RMT \ = \  \frac{1}{2} \sum_{\ml{U} \subseteq \{1, \ldots, n\}} 2^{|\ml{U}|} C(\ul{O}_{\ml{U}})
\sum_{\underline{G} \in \Pi(\ml{U}^c)} \sum_{\ul{F} \preceq \ul{G}}^{\op{2ref}} 2^{|\ml{U}^c| - \nu(\underline{G})} \mu(\ul{O},\ul{F}) D(\ul{F},\ul{G}) \prod_{l \in S^c} A_l,
\ee
which is the desired expression.
\end{proof}

\subsubsection{The Fourier identity}

Lemma \ref{lem:fourier_forms} reduces the density conjecture to showing that the two expressions \eqref{new_RMT_fourier} and \eqref{new_NT_fourier} are equal. These expressions separate the $C$ and $E$ terms from the others, so the question is: how do they match up? We believe the following conjecture, which essentially says that they match term-by-term in the form given by Lemma \ref{lem:fourier_forms}.

\begin{conjecture}[Fourier Identity 1] \label{conj:fourier_identity_1}
With notation as in Lemma \ref{lem:fourier_forms}, and $\ast$ denoting (incidence algebra) convolution, $C$ is the Mobius transform of $E$:
\be \label{fourier_identity_simple}
C \ = \  \mu \ast E, \text{ or equivalently } \zeta \ast C\ =\ E,
\ee as functions on $\Pi(n)$.
\end{conjecture}
In particular, the identity we need, which we apply once for each subset $\ml{U} \subseteq \{1, \ldots, n\}$ in \eqref{new_RMT_fourier} and \eqref{new_NT_fourier}, is simply
\begin{conjecture}[Fourier Identity 2]  \label{conj:fourier_identity_2}
With notation as in Lemma \ref{lem:fourier_forms},
\be \label{fourier_identity_1}
C(\ul{O}) \ = \  \sum_{\ml{T} \in \Pi(n)} \mu(\ul{O},\ul{\ml{T}}) E(\ul{\ml{T}}).
\ee
Equivalently, by Mobius inversion,
\be \label{fourier_identity_2}
\sum_{\ul{\ml{T}} \in \Pi(n)} C(\ul{\ml{T}}) \ = \  E(\ul{O}).
\ee
\end{conjecture}
It is clear that Conjecture  \ref{conj:fourier_identity_1} implies Conjecture  \ref{conj:fourier_identity_2} and Conjecture \ref{conj:fourier_identity_2} implies the Density Conjecture. In fact, Conjecture \ref{conj:fourier_identity_2} is equivalent to Conjecture \ref{conj:fourier_identity_1}. The equivalence stems from the fact that if $\ul{F} \in \Pi(n)$ has $k$ blocks $F_1, \ldots, F_k$, then $C(\ul{F})$ and $E(\ul{F})$ are identical to the integrals $C(\ul{O}), E(\ul{O})$ for $\ul{O} \in \Pi(k)$, using the $F_i(x)$ as a new set of test functions. Thus, if the identity \eqref{fourier_identity_2} holds for $k \leq n$ (for all choices of test function), so does the identity \eqref{fourier_identity_1}.

Note that equation \eqref{fourier_identity_1} is the result of taking the equations \eqref{new_RMT_fourier} and \eqref{new_NT_fourier}, assuming inductively that the identity holds for $< n$, and discarding all matching terms. As such, it is actually equivalent to the Density Conjecture (and, therefore, must be true by the results of \cite{ER-GR}, though there should be a purely combinatorial proof of this fact).

For our purposes, identity \eqref{fourier_identity_2} is preferable since, in contrast to \eqref{fourier_identity_1}, all the summands are easy to convert to integrals over the same region ($\R^n_{\geq 0}$), which we do in \S\ref{sec:fourier_identity_n6}. We summarize our results so far.

\begin{theorem}[Reduction to Fourier Identity] \label{thm:impliedbyfourier}
Let $f_1, \ldots, f_n$ be even test functions with $\widehat{f_1}, \ldots, \widehat{f_n}$ supported in $\sum_{i=1}^n |u_i| < 2$. The Fourier identity \eqref{fourier_identity_2} implies the Density Conjecture for quadratic Dirichlet $L$-functions,
\be\int_{\R^n} \prod_{i=1}^n f_i(x) W^{(n)}_{USp}(x)dx \ = \  \int_{\R^n} \prod_{i=1}^n f_i(x) W^{(n)}_{Q}(x)dx.\ee
\end{theorem}

In the next section, we study the Fourier identity further and put it in a `canonical' form, which we use to verify the cases $n \leq 7$. We prove

\begin{theorem}[Density Conjecture, $n \leq 7$] \label{thm:density_n6}
With notation and assumptions as in Theorem \ref{thm:impliedbyfourier}, the Fourier Identity \eqref{fourier_identity_2} holds for $n \leq 7$. \end{theorem}

The above immediately implies our main result, Theorem \ref{thm:mainnlevelagree}. In particular, for $n \leq 7$ the $n$-level density of zeros of quadratic Dirichlet $L$-functions $\{L(s,\chi_{8d})\}$ (with $d \in \N$ odd and square-free) is the same as the $n$-level eigenvalue density of the Unitary Symplectic Ensemble (for test functions where the sum of the supports is at most 2).

The remainder of the paper is devoted to the proof of the Fourier Identity and Theorem \ref{thm:density_n6}.


\section{The Fourier Identity} \label{sec:fourier_identity_n6}

In this section we consider the Fourier Identity \eqref{fourier_identity_2} for a fixed $n$:
\be \label{fourier_identity_2_repeat}
\sum_{\ul{F} \in \Pi(n)} C(\ul{F}) \ = \  E(\ul{O}),
\ee
where for $\ul{F} \in \Pi(n)$,
\begin{align}
\label{sec4:new_cterm}
C(\ul{F}) &\ = \  \frac{1}{2} \int_{\R^{\nu(\ul{F})}} \bigg(\mu(\ul{F},\ul{N}) + (-1)^{\nu(\ul{F})} \chi^*_{\nu(\ul{F})}(u_1,\ldots,u_{\nu(\ul{F})})  \bigg) \prod_{l = 1}^{\nu(\ul{F})} \widehat{F_l}(u_l)du_l, \\
\label{sec4:new_eterm}
E(\ul{O}) &\ = \  \int_{\R_{\geq 0}^n} \bigg( \sum_{I \subseteq \{1, \ldots,n\}} (-1)^{|I|+1} \tilde\chi(\sum_I u_i - \sum_{I^c} u_i ) \bigg) \prod_{i=1}^n \widehat{f_i}(u_i)du_i;
\end{align}
see Definition \ref{appendix:chistar} for the definition of  $\chi^*_{\nu(\ul{F})}$.

We first reduce the Fourier Identity to a canonical form. Our method is to convert all the summands $C(\ul{F})$ to integrals over $\R^n_{\geq 0}$ and reduce to a sum of products of indicator functions of the form
\be\tilde \chi(\vep_1 u_1 + \cdots + \vep_n u_n), \qquad \tilde\chi \ = \  \mathbb{I}_{(1,\infty)}, \ee
with each $\vep_i = \pm 1$. We examine, for each term, the set $\{i : \vep_i = +1\}$, and use combinatorial arguments and the assumption $\sum \supp(\widehat{f_i}) < 2$ to simplify some terms and show that others are identically zero. We start by reducing to a canonical form:

\begin{proposition}[Fourier identity, canonical form] \label{lem:intermediate_fourierone}
With notation as above,
\be \label{fourier_c_integrals}
\sum_{\ul{F} \in \Pi(n)} C(\ul{F}) \ = \  \int_{\R^n_{\geq 0}} C_0(u_1, \ldots, u_n) \prod_{i=1}^n \widehat{f_i}(u_i)du_i,
\ee
where
\be \label{fourier_canonical_chains}
C_0(u_1, \ldots, u_n) \ = \  \sum_{J \subseteq \{2, \ldots, n\}} \sum_{\substack{\text{{\rm chains} } \ml{A} \\ A_1 \ = \  \{1\}}} (-1)^{n-k} \prod_{i=1}^k (\tilde{\chi}_{A_i \triangle J} + \tilde{\chi}_{A_i^c \triangle J}),
\ee
where $\ml{A}$ ranges over the chains $A_1 \subset \cdots \subset A_k$, such that $A_1 = \{1\}$, and for a subset $W \subseteq \{1, \ldots, n\}$, we write
\be\tilde \chi_W :\ = \  \tilde \chi\big( \sum_{i \in W} u_i - \sum_{i \notin W} u_i \big), \qquad \tilde\chi \ = \  \mathbb{I}_{(1,\infty)}, \ee
and $\triangle$ denotes symmetric difference.
\end{proposition}
Hence, comparing \eqref{sec4:new_eterm} and \eqref{fourier_c_integrals}, the identity of functions
\be \label{fourier_canonical_chains_2}
C_0(u_1, \ldots, u_n) \ = \  \sum_{I \subseteq \{1, \ldots, n\}} (-1)^{|I|+1} \tilde \chi_I
\ee
on the simplex $\{u_i > 0; \sum u_i < 2\} \subset \R^n$ implies the Fourier identity \eqref{fourier_identity_2}.

While we are unable to prove \eqref{fourier_canonical_chains_2} for all $n$, we give a method for checking it for specific values of $n$. Our method is partly ad hoc, but suffices for $n \leq 7$ (for larger $n$, the computations are the same, but become intractable). We note that for $n=1, 2, 3$, the identities \eqref{fourier_identity_2_repeat} appear in \cite{Gao} as, respectively, the last (unnumbered) equations on pages 57 and 58, and equation (5.7).

\subsection{The canonical form}
We first reduce \eqref{fourier_identity_2_repeat} to a sum of indicator functions by converting all integrals to the region $\R^n_{\geq 0}$.

\begin{lemma} \label{lem:fourier_identity_indicators}
With notation as above, the Fourier Identity \eqref{fourier_identity_2_repeat} for $n>1$ follows from the equality of indicator functions
\be \label{fourier_identity_indicators}
\sum_{\ul{F} \in \Pi(n)} C'(\ul{F}) \ = \  E'(\ul{O})
\ee
on the region $\{0 < u_i < \supp(\widehat{f_i}),\  i = 1, \ldots, n) \} \subset \R^n_{\geq 0}$, where
\begin{align}
\label{cterm_indicators}
C'(\ul{F}) &\ = \  \frac{1}{2} \sum_{\substack{\vep_i = \pm 1, \\ i=1, \ldots, n}} \bigg( \mu(\ul{F},\ul{N}) + (-1)^{\nu(\ul{F})} \chi^*_{\nu(\ul{F})}\big(\sum_{i \in F_1}\vep_i u_i,\ldots,\sum_{i \in F_{\nu(\ul{F})}} \vep_i u_i \big) \bigg), \\
\label{eterm_indicators}
E'(\ul{O}) &\ = \  \sum_{I \subseteq \{1, \ldots, n\}} (-1)^{|I|+1} \tilde\chi(\sum_I u_i - \sum_{I^c} u_i ).
\end{align}
(Note that $E'(\ul{O})$ is just the sum of indicator functions in the integrand of $E(\ul{O})$.)
\end{lemma}


\begin{proof}
Our test functions $f_1, \ldots, f_n$ are all even, so we have two identities. First, for any partition $\ul{F} \in \Pi(n)$ having $k$ blocks,
\be
\int_{\R^k} g(u_1, \ldots, u_k) \prod_{i=1}^k \widehat{F_i}(u_i)du_i \ = \  \int_{\R^n} g\big( \sum_{i \in F_1}u_i,\ldots,\sum_{i \in F_{\nu(\ul{F})}} u_i \big) \prod_{i=1}^n \widehat{f_i}(u_i)du_i
\ee
holds for any integrand $g$ via a linear change of variables. Second,
%
\be
\int_{\R^n} h(u_1, \ldots, u_n) \prod_{i = 1}^{n} \widehat{f_i}(u_i)du_i \ = \  \sum_{\substack{\vep_i = \pm 1 \\ i=1, \ldots, n}} \int_{\R^n_{\geq 0}} h(\vep_1 u_1, \ldots, \vep_n u_n) \prod_{i = 1}^{n} \widehat{f_i}(u_i)du_i
\ee
holds for any $h$.  Applying these transformations gives, with notation as in \eqref{cterm_indicators},
\be \label{desym_cterm}
C(\ul{F}) \ = \  \int_{\R^n_{\geq 0}} C'(\ul{F}) \cdot \prod_{i=1}^n \widehat{f_i}(u_i)du_i,
\ee
which is the desired expression.
\end{proof}


We use the identity
\be
\chi(u) \ = \  1 - \tilde{\chi}(u) - \tilde{\chi}(-u)
\ee
to rewrite the left-hand side as a sum of products of terms of the form $\tilde{\chi}(\sum \vep_i u_i)$ with each $\vep_i = \pm 1$. The advantage of using $\tilde{\chi}$ throughout comes from not being an even function: to know whether $\chi(\sum \vep_i u_i) = 0$, we need to consider both $\sum \vep_i u_i > 1$ and $\sum \vep_i u_i < -1$, but with $\tilde{\chi}$ only the first case matters. This will facilitate several simplifications.

\begin{definition}[Combinatorial notation] We adopt the following notation: given a term $\tilde{\chi}(\vep_1u_1 + \cdots + \vep_n u_n),$ let $A$ be the set of indices for which $\vep_i = +1$ and $A^c$ the set for which $\vep_i = -1$. We define
\be
\tilde{\chi}_A\ :=\ \tilde{\chi}(\sum_A u_{a_i} - \sum_{A^c} u_{a_i}).
\ee
\end{definition}
This notation reduces arguments about products of $\tilde{\chi}(\sum_i \vep_i u_i)$ to combinatorial arguments about subsets $A \subseteq \{1, \ldots, n\}$. The $\chi^*$ integrand for a partition $\ul{F} \in \Pi(n)$ is thus, in combinatorial notation,
\be
\chi^*(\ul{F}) \ = \  \sum_{\substack{ \pi \in S_{\nu(\ul{F})} \\ \pi(1)=1 }} \prod_{i=1}^{\nu(\ul{F})} \big(1 -\tilde{\chi}_{F_{\pi(1)} \cup \cdots \cup F_{\pi(i)}} - \tilde{\chi}_{(F_{\pi(1)} \cup \cdots \cup F_{\pi(i)})^c} \big),
\ee
where $S_{\nu(\ul{F})}$ is the group of permutations of $\{1, \ldots, \nu(\ul{F})\}$.

Changing the signs of some of the $\vep_i$ in $\tilde \chi_A$ is equivalent to taking a symmetric difference, replacing $\tilde{\chi}_A$ with $\tilde{\chi}_{A \triangle J}$, where $J \subseteq \{1, \ldots, n\}$ is the set of indices whose signs have been changed. We note that $\mu(\ul{F},\ul{N}) = (-1)^{\nu(\ul{F})-1}(\nu(\ul{F})-1)!$ is the same as the number of permutations on the inner sum. We thus write
\begin{align} \nonumber
C'(\ul{F}) &\ = \  \frac{1}{2} \sum_{J \subseteq \{1, \ldots, n\}}\bigg(\mu(\ul{F},\ul{N}) \nonumber\\ & \ \ \ \ \ \ \ \ \ \ \ \ + (-1)^{\nu(\ul{F})} \sum_{\substack{ \pi \in S_{\nu(\ul{F})} \\ \pi(1)=1 }} \prod_{i=1}^{\nu(\ul{F})} \big(1 -\tilde{\chi}_{(F_{\pi(1)} \cup \cdots \cup F_{\pi(i)}) \triangle J} - \tilde{\chi}_{(F_{\pi(1)} \cup \cdots \cup F_{\pi(i)}) \triangle J^c} \big) \bigg)\nonumber\\
%
\label{cterm_combo_form}
&\ = \  \frac{1}{2} \sum_{J \subseteq \{1, \ldots, n\}} \sum_{\substack{ \pi \in S_{\nu(\ul{F})} \\ \pi(1)=1 }} \bigg( (-1)^{\nu(\ul{F})-1}\nonumber\\ & \ \ \ \ \ \ \ \ \ \ \ \ + \prod_{i=1}^{\nu(\ul{F})} \big(\tilde{\chi}_{(F_{\pi(1)} \cup \cdots \cup F_{\pi(i)}) \triangle J} + \tilde{\chi}_{(F_{\pi(1)} \cup \cdots \cup F_{\pi(i)}) \triangle J^c} - 1 \big) \bigg).
\end{align}
In combinatorial notation, the right-hand side of \eqref{fourier_identity_indicators} is just
\be \label{eterm_combo_form}
E'(\ul{O}) \ = \  \sum_{I \subseteq \{1, \ldots, n\}} (-1)^{|I|+1} \tilde{\chi}_I.
\ee

To shorten the notation, we write the summands in terms of chains rather than partitions.
%
%
Given a partition $\ul{F} \in \Pi(n)$ and a permutation $\pi \in S_{\nu(\ul{F})}$, such that $\pi(1)=1$, we obtain an ascending chain
\be F_{\pi(1)} \subset F_{\pi(1)} \cup F_{\pi(2)} \subset \cdots \subset F_{\pi(1)} \cup \cdots \cup F_{\pi(k)}\ =\ \{1, \ldots, n\}.\ee
Thus each choice of $\ul{F}$ and $\pi$ corresponds uniquely to a strictly ascending chain
\be
\ml{A} : A_1 \subsetneq A_2 \subsetneq \cdots \subsetneq A_k
\ee of subsets of $\{1, \ldots, n\}$ such that $1 \in A_1$ and $A_k = \{1, \ldots, n\}$. The corresponding product of indicator functions is then
\be
(\tilde{\chi}_{A_1 \triangle J} + \tilde{\chi}_{A_1^c \triangle J} - 1) (\tilde{\chi}_{A_2 \triangle J} + \tilde{\chi}_{A_2^c \triangle J} - 1) \cdots (\tilde{\chi}_{A_k \triangle J} + \tilde{\chi}_{A_k^c \triangle J} - 1).
\ee
Thus we can write the left-hand side of \eqref{fourier_identity_indicators} as
\be \label{cterm_chains_form}
\sum_{\ul{F} \in \Pi(n)} C'(\ul{F}) \ = \  \frac{1}{2} \sum_{J \subseteq \{1, \ldots, n\}} \sum_{\substack{\text{chains } \ml{A}, 1 \in A_1 \\ A_k \ = \  \{1, \ldots, n\}}} \bigg( (-1)^{k-1} + \prod_{i=1}^k (\tilde{\chi}_{A_i \triangle J} + \tilde{\chi}_{A_i \triangle J^c} - 1) \bigg),
\ee
where $\ml{A}$ ranges over all ascending chains of subsets of $\{1, \ldots, n\}$ such that $1 \in A_1$ and $A_k = \{1, \ldots, n\}$ is the last (largest) set in the chain.

Observe that a given product of $\tilde\chi$ terms occurs many times in the sum \eqref{cterm_chains_form} when the $(-1)$ factors are expanded. We account for this cancelation below and give the canonical form of the Fourier Identity. We employ the following standard facts about chains.

\begin{proposition}[Sums over chains] \label{lem:chain_sum}
Let $\ml{A}$ and $\ml{B}$ be chains. If $\ml{A} = A_1 \subset \cdots \subset A_k$, write $k = |\ml{A}|$, and if $\ml{B}$ is a subchain of $\ml{A}$, write $\ml{A} \succeq \ml{B}$. For a fixed $\ml{B}$ with $1 \in B_1$,
\be
\displaystyle{\sum_{\substack{\ml{A} \succeq \ml{B}, 1 \in A_1, \\ A_k = \{1, \ldots, n\}}} (-1)^{|\ml{A}|} \ = \  \begin{cases} (-1)^n & \text{ {\rm if} } B_1 = \{1\}, \\ 0 & \text{ {\rm otherwise}.}\end{cases}}
\ee
\end{proposition}

See appendix \ref{appendix:chains} for a proof.

\begin{lemma}[Fourier identity, canonical form] \label{lem:intermediate_fouriertwo}
With notation as above,
\be \label{intermediate_fouriereq}
\sum_{\ul{F} \in \Pi(n)} C'(\ul{F}) \ = \  \sum_{J \subseteq \{2, \ldots, n\}} \sum_{\substack{\text{{\rm chains} } \ml{A} \\ A_1 \ = \  \{1\}}} (-1)^{n-|\ml{A}|} \prod_{i=1}^{|\ml{A}|} (\tilde{\chi}_{A_i \triangle J} + \tilde{\chi}_{A_i^c \triangle J}),
\ee
where $\ml{A}$ ranges over the chains $A_1 \subset \cdots \subset A_k$ ($k = |\ml{A}|$) such that $A_1 = \{1\}$. (We do not require $A_k = \{1, \ldots, n\}$.)
\end{lemma}

\begin{proof}
Consider the expansion for $\sum_{\ul{F} \in \Pi(n)} C'(\ul{F})$ in \eqref{cterm_chains_form}. Since the right hand side is invariant under interchanging $J$ and $J^c$, we replace $\frac{1}{2} \sum_{J \subseteq \{1, \ldots, n\}}$ with $\sum_{J \subseteq \{2, \ldots, n\}}$; that is, we may assume without loss of generality that $1 \notin J$.

Now we expand the $(-1)$ factors in \eqref{cterm_chains_form}. The $(-1)^k$ cancels with the $(-1)^{k-1}$, so we are left with
\begin{align}
\sum_{\ul{F} \in \Pi(n)} C'(\ul{F}) &\ = \  \sum_{J \subseteq \{2, \ldots, n\}} \sum_{\substack{\text{chains } \ml{A}, 1 \in A_1 \\ A_k \ = \  \{1, \ldots, n\}}} \sum_{\substack{ W \subseteq \{1, \ldots, t\} \\ W \ne \eset}} (-1)^{|\ml{A}| - |W|} \prod_{i \in W} (\tilde{\chi}_{A_i \triangle J} + \tilde{\chi}_{A_i \triangle J^c}) \nonumber\\
&\ = \  \sum_{J \subseteq \{2, \ldots, n\}} \sum_{\substack{\text{chains } \ml{A}, 1 \in A_1 \\ A_k \ = \  \{1, \ldots, n\}}} \sum_{\substack{\text{chains } \ml{B} \preceq \ml{A} \\ \ml{B} \ne \eset}} (-1)^{|\ml{A}| - |\ml{B}|} \prod_{\ml{B}} (\tilde{\chi}_{B_i \triangle J} + \tilde{\chi}_{B_i \triangle J^c}),
\end{align}
where $\ml{B}$ ranges over the subchains of $\ml{A}$ (excluding the `empty chain' with no sets). We switch orders of summation on $\ml{B}$ and $\ml{A}$. We have
\be
\sum_{\substack{\text{chains } \ml{A}, 1 \in A_1 \\ A_k \ = \  \{1, \ldots, n\}}} \sum_{\eset \ne \ml{B} \preceq \ml{A}} \ = \  \sum_{\substack{ \ml{B} \\ 1 \in B_1}} \sum_{\substack{ \ml{A} \succeq \ml{B}, 1 \in A_1, \\ A_k \ = \  \{1, \ldots, n\}}},
\ee
and so
\begin{align}
\sum_{\ul{F} \in \Pi(n)} C'(\ul{F}) &\ = \  \sum_{J \subseteq \{2, \ldots, n\}} \sum_{\substack{\text{chains } \ml{B} \\ 1 \in B_1}}  (-1)^{|\ml{B}|} \prod_{\ml{B}} (\tilde{\chi}_{B_i \triangle J} + \tilde{\chi}_{B_i \triangle J^c}) \Big( \sum_{\substack{\ml{A} \succeq \ml{B}, 1 \in A_1, \\ A_k \ = \  \{1, \ldots, n\}}} (-1)^{|\ml{A}|} \Big) \nonumber\\
&\ = \  \sum_{J \subseteq \{2, \ldots, n\}} \sum_{\substack{\text{chains } \ml{B} \\ B_1 \ = \  \{1\}}} (-1)^{n-|\ml{B}|} \prod_{\ml{B}} (\tilde{\chi}_{B_i \triangle J} + \tilde{\chi}_{B_i \triangle J^c})
\end{align}
by Proposition \ref{lem:chain_sum}.
\end{proof}

\subsection{Breaking down the combinatorics}
We describe our approach to confirm the Fourier identity \eqref{fourier_identity_indicators} in the cases $n \leq 7$. These arguments are impractical to do by hand for $n \geq 4$; we ran them in Mathematica with code available at
 \begin{itemize} \item \url{http://www-personal.umich.edu/~jakelev/} \item \tiny \url{http://web.williams.edu/Mathematics/sjmiller/public_html/math/papers/jakel/FourierIdentity.tar}.\normalsize \end{itemize}



The simplifications we use are as follows.

\begin{lemma}[Simplifications] \label{lem:simplifications}
Let $A, B \subset \{1,\ldots,n\}$. Then
\begin{align}
\label{subsetsimp}
\tilde{\chi}_A \cdot \tilde{\chi}_B &\ = \  \tilde{\chi}_A \text{ {\rm whenever} } A \subset B.
\end{align}
Let $A_1, \ldots, A_k \subseteq \{1, \ldots, n\}$, with $k \geq 2$. For each $i \in \{1, \ldots, n\}$, let $e_i$ be the number of the $A_j$'s that contain $i$. Then
\begin{align}
\label{addineqsimp}
\tilde{\chi}_{A_1} \cdots \tilde{\chi}_{A_k} &\ = \  0 \text{ {\rm if} }\ e_i \leq \tfrac{3}{4}k\ \text{ {\rm for each} }\ i \in \{1, \ldots, n\}.
\end{align}
\end{lemma}
The first equation says, equivalently, that given a product $\tilde{\chi}_{A_1} \cdots \tilde{\chi}_{A_k},$ we need only keep the $\tilde{\chi}_{A_j}$'s for which the subsets $A_j \subset \{1, \ldots, n\}$'s are minimal with respect to containment, i.e., the $\tilde{\chi}$ terms having few positive signs. The identity is essentially a formal sum of antichains, with additional relations such as \eqref{addineqsimp}. We remark also that for $k=2$, \eqref{addineqsimp} is just the statement $\tilde \chi_A \cdot \tilde \chi_B = 0$ when $A \cap B = \eset$.

\begin{proof}[Proof of \eqref{subsetsimp}] If $A \subset B$, then $B^c \subset A^c$, so we have the inequalities
\be
\sum_B u_{b_i}\ \geq\ \sum_A u_{a_i}\ \  \text{ and }\ \  \sum_{A^c} u_{a_i}\ \geq\ \sum_{B^c} u_{b_i}.
\ee

Combining these yields that whenever $\sum_A u_{a_i} - \sum_{A^c} u_{a_i} > 1$, we also have
\be
\sum_B u_{b_i} - \sum_{B^c} u_{b_i}\ \geq\ \sum_A u_{a_i} - \sum_{A^c} u_{a_i}\ \geq\ 1.
\ee

So if $\tilde{\chi}_A = 1$, it follows that $\tilde{\chi}_B = 1$ (if $\tilde{\chi}_A = 0$, then both sides of \eqref{subsetsimp} are 0).
\end{proof}

\begin{proof}[Proof of \eqref{addineqsimp}] Add the inequalities $\sum_{a \in A_j} u_a - \sum_{a' \in A_j^c} u_{a'} > 1$ together. Then $+u_i$ occurs $e_i$ times and $-u_i$ occurs $k-e_i$ times, so the result is
\be
(2e_i-k) u_1 + \cdots + (2e_n-k) u_n\ >\ k.
\ee
The condition $e_i \leq \tfrac{3}{4}k$ is the same as $\tfrac{k}{2} \geq 2e_i - k$, yielding
\be
\tfrac{k}{2}(u_1 + \cdots + u_n)\ >\ k,
\ee
that is, $u_1 + \cdots + u_n > 2$, violating the support restriction.
\end{proof}

\begin{rmk}
For $k \leq 4$, the second condition \eqref{addineqsimp} is equivalent to $\bigcap_{i=1}^k A_i = \eset$. For $k > 4$ it is a stronger condition.
\end{rmk}

In sum, our method of verifying the Fourier identity \eqref{fourier_identity_indicators} is to apply the simplifications above to the sum \eqref{intermediate_fouriereq} to simplify and remove terms. We are able to verify the cases $n\leq 7$ this way; for $n = 8$ the verification becomes intractable, since the number of terms in the left-hand side of \eqref{fourier_identity_indicators} becomes large ($2^n \cdot$ sequence A027882 in the Online Encylopedia of Integer Sequences).

Although we cannot prove the identity for all $n$, we give one last conjecture that indicates one way of grouping terms in the identity.

\begin{conjecture} \label{conj:simplify_fourier}
For fixed $n$ and $J \subseteq \{2, \ldots, n\}$, let
\begin{align} \label{J_sum_simplified}
\op{simp}_1(J;n) &\ = \  \bigg( \sum_{A \subseteq \{J\}} (-1)^{|A|} \tilde\chi_{A \cup \{1\}} \bigg) \cdot
\bigg( \sum_{B : J \subseteq B \subseteq \{2, \ldots, n\}} (-1)^{|B| - |J| - 1} \tilde \chi_B \bigg), \nonumber\\
\op{simp}_2(n) &\ = \  \sum_{A \subseteq \{2, \ldots, n\}} (-1)^{|A|} \tilde \chi_{A \cup \{1\}}.
\end{align}
Then the inner sum of the Fourier identity \eqref{intermediate_fouriereq} is
\be
\sum_{\substack{\text{chains } \ml{A} \\ A_1 \ = \  \{1\}}} (-1)^{n-|\ml{A}|} \prod_{i=1}^{|\ml{A}|} (\tilde{\chi}_{A_i \triangle J} + \tilde{\chi}_{A_i^c \triangle J}) \ = \
\begin{cases}
\op{simp}_1(J;n) & \text{ if } J \ne \{2, \ldots, n\} \\
\op{simp}_1(J;n) + \op{simp}_2(n) & \text { if } J = \{2, \ldots, n\}.
\end{cases}
\ee
\end{conjecture}
It is easy to see, by summing over $J \subseteq \{2, \ldots, n\}$, that this conjecture implies the Fourier identity. The identity is easily checked for $J = \eset, \{2\}, \{2, \ldots, n\}$; for the remaining cases, it is sufficient (by relabeling) to consider $J = \{2, \ldots, i\}$ for $3 \leq i \leq n-1$, but we do not as yet have a proof.




\section{Concluding remarks} \label{sec:concluding}

By adopting an appropriate combinatorial perspective, we are able to unify the analysis of the number theory and random matrix theory expansions. We reduce showing agreement of the two expressions of the $n$-level density to a combinatorial identity, which we can verify for $n \le 7$. As there should be a purely combinatorial proof of this identity, we conclude with a few thoughts related to it; we welcome any correspondence with people interested in extending these arguments.



\subsection{Verifying the identity formally} \label{sec:verify_formal}
We can view the Fourier Identity as a formal identity: the indicator functions $\tilde \chi_A$ generate a subring $C(n) \subseteq \ml{L}_{\infty}(\R^n)$ that is a quotient of a polynomial ring in $2^n$ variables,
\be \C[x_A : A \subseteq \{1, \ldots, n\}] \to C(n), \qquad x_A \mapsto \tilde \chi_A \ee
As a ring of functions, $C(n)$ is certainly reduced, so it is sufficient to check that the identity holds over every quotient $C(n)/P$, where $P \in \op{Spec} C(n)$ is a prime ideal.

By equation \eqref{subsetsimp} of Lemma \ref{lem:simplifications}, this map factors through the quotients
\be C'(n) = \frac{\C[x_A : A \subseteq \{1, \ldots, n\}]}{( x_A x_B - x_A : A \subseteq B )}, \qquad C''(n) = \frac{C(n)}{\sum ( x_{A_1} \cdots x_{A_k} ) }, \ee
where the second quotient is by the monomials $x_{A_1} \cdots x_{A_k}$ such that $\tilde \chi_{A_1} \cdots \tilde \chi_{A_k}$ is identically zero as an indicator function in the supported region, namely
\be \{u_1 > 0, \ldots, u_n > 0, \sum u_i < 2\} \subseteq \R^n. \ee
(We remark that condition \eqref{addineqsimp} does not describe all such products.)

For $C'(n)$, prime ideals are in one-to-one correspondence with antichains: if $\{W_1, \ldots, W_k\}$ is an antichain, the corresponding prime ideal is
\be P\ =\ (x_A : \text{ for each } i,  \ A \not \supseteq W_i) + (x_A - 1 : \text{ for some } i, A \supseteq W_i).\ee
Passing to $C''(n)$ just removes `identically-zero' antichains from consideration. For each of the remaining antichains $\ml{W} = \{W_1, \ldots, W_k\}$, we consider the Fourier Identity under the map $C(n) \to \C$ that sends
\be \label{fourier_evaluate}
x_A \mapsto \begin{cases} 1 & A \supseteq W_i \text{ for some } i, \\ 0 & \text{ otherwise}. \end{cases}
\ee
Verifying that the Fourier identity holds under each of these maps is sufficient to verify the full Fourier identity. Assuming $C''(n) \cong C(n)$ (that is, assuming there are no additional relations between the $\tilde \chi_A$), this is also a necessary condition.



We express both sides of the Fourier identity in terms of Euler characteristics. Fix an antichain $\ml{W} = \{W_1, \ldots, W_k\}$ and let
\be
S\ =\  S(\ml{W}) \ =\ \{A \subseteq \{1, \ldots, n\} : A \subseteq W_i^c \text{ for some } i\}
\ =\ \bigcup_{i=1}^k [\eset, W_i^c].
\ee
Thus $S$ is a simplicial set; its vertices are $\bigcup_{i=1}^k W_i^c$ and its maximal faces are the $W_i^c$. By evaluating as in \eqref{fourier_evaluate}, the sets $A$ with $A^c \in S$ evaluate to 1, so the right-hand side of the Fourier identity becomes
\be \sum_{A : A^c \in S} (-1)^{|A|+1}\ =\ \sum_{A \in S} (-1)^{n - |A| + 1}\ =\ (-1)^{n-1} \chi_{\rm Eul}(S), \ee
up to a sign, the Euler characteristic of the simplicial complex $S$.

We now express the left-hand side in a related way. First, we determine the value of $\tilde \chi_{A \triangle J} + \tilde \chi_{A^c \triangle J}$ under the evaluation map. Given $J, W \subseteq \{1, \ldots, n\}$, let $U(J,W)$ be the union of segments from $\ml{P}(1, \ldots, n)$,
\be
U(J,W)\ =\ \big[W - J, W^c \cup (W - J) \big] \cup \big[W \cap J, W^c \cup (W \cap J) \big].
\ee
It is easy to see the following:
\begin{enumerate}
\item the two segments are disjoint if $W \ne \eset$,
\item $A \triangle J \supseteq W$ if and only if $A$ is in the first segment,
\item $A^c \triangle J \supseteq W$ if and only if $A$ is in the second segment.
\end{enumerate}
In particular, we conclude that, evaluated at $\ml{W}$,
\be \tilde \chi_{A \triangle J} + \tilde \chi_{A^c \triangle J}\ =\ \begin{cases} 1 & A \in \bigcup_{i=1}^k U(J,W_i) \\ 0 & \text{ otherwise}. \end{cases} \ee
We are only interested in chains where $A_1 = \{1\}$, so let
\begin{align}
U(J;\ml{W}) &\ =\ \bigcup_{i=1}^k U(J,W_i), \\
\widehat{U} (J;\ml{W}) &\ =\ \big(\{1\},\{1, \ldots, n\} \big] \cap U(J;\ml{W}).
\end{align}
Evaluating the inner summand of the Fourier Identity for $J$ gives a sum over all chains in $U$ that begin with $A_1 = \{1\}$. If $\{1\} \notin U$, there are no such chains in $\ml{U}$, so the sum is 0; otherwise, such chains are in bijection with \emph{all} chains $\ml{A}'$ in $\widehat{U}$ (with the length off by 1), so that
\be \sum_{\substack{\text{chains } \ml{A} \\ A_1 \ = \  \{1\}, \\ A_i \in U}} (-1)^{n-|\ml{A}|}\ =\ (-1)^n \sum_{\substack{ \text{chains } \ml{A}' \\ A_i \in \widehat{U}}} (-1)^{|\ml{A}'|-1}\ =\ (-1)^n \chi_{\rm Eul}(\widehat{U}),
\ee
the Euler characteristic of the order complex $\Delta_{\rm ord}(\widehat{U})$.


In other words, the Fourier Identity now reads, with the $(-1)^n$ canceled and $\chi_{\rm Eul}$ denoting the Euler characteristic,
\be \sum_{\substack{ J \subseteq \{2, \ldots, n\} \\ \{1\} \in U(J;\ml{W})}} -\chi_{\rm Eul}(\widehat{U}(J;\ml{W}))\ =\ \chi_{\rm Eul}(S(\ml{W})). \ee

\subsection{Which products of the indicator functions $\tilde \chi_A$ are 0?}

In addition to the approach using Euler characteristics in section \ref{sec:verify_formal}, the authors are interested in suggestions or answers to the problem of determining which products $\tilde \chi_{A_1} \cdots \tilde \chi_{A_k}$ are identically zero in the integration region, where $A_j \subseteq \{1, \ldots, n\}$ and
\[\tilde \chi_A := \tilde \chi \big( \sum_{i \in A} x_i - \sum_{i \in A^c} x_i),\]
and $\tilde \chi(x)$ is the indicator function of $(1,\infty)$.

In other words, we wish to solve the following linear program: let $M$ be an $k\times n$ matrix with each entry $\pm 1$, and let
\begin{align}
b &= (1 \cdots 1)^T \in \R^k, \\
c &= (1 \cdots 1)^T \in \R^n.
\end{align}
Minimize $c^T x = \sum_i x_i$, subject to
\begin{align}
Mx &\geq b, \\
x &\geq 0.
\end{align}
For $j = 1, \ldots, k$, let $A_j \subseteq \{1, \ldots, n\}$ be the set of $+1$'s in the $j$\textsuperscript{th} row of $M$. Then the product $\tilde \chi_{A_1} \cdots \tilde \chi_{A_k}$ is identically zero iff one of the following holds:
\begin{enumerate}
\item the minimum of $c^T x$ is 2 or greater, or
\item the problem is infeasible.
\end{enumerate}
The product is nonzero iff the minimum $c_* \in [0,2)$. Note that the problem cannot be unbounded since $c^Tx \geq 0$. Of course, we could replace the objective function by the inequality $\sum_i x_i < 2$.

\subsection*{Acknowledgements}
The authors were partially supported by NSF grant DMS0970067. The description of ${\rm Spec}\ C(n)$ and the observation that, because $C(n)$ is reduced, the Fourier Identity can be checked at each prime ideal of $C(n)$ separately, is due to David Speyer. We thank him for this and other related discussions on the identity. We also thank \zeev\ Rudnick for sharing slides on their preprint \cite{ER-GR}, and the referee for many helpful comments.

\appendix

\section{Sums over chains} \label{appendix:chains}

We give the proof of Lemma \ref{lem:chain_sum} involving sums over chains as \ref{sums-chains}(3) below. The authors thank B. Ullery for the proof of \ref{sums-chains}(1).

\begin{lemma} \label{sums-chains}
Given chains $\ml{A}, \ml{B}$, we write $\ml{B} \preceq \ml{A}$ if $\ml{B}$ is a subchain of $\ml{A}$ (we include the `empty chain' with no sets). If $\ml{A} = A_1 \subset \cdots \subset A_k$, we write $k = |\ml{A}|$. Then
\begin{enumerate}
\item For any $n$, $\displaystyle{\sum_{\substack{\ml{A}: 1 \in A_1, \\ A_k \ = \  \{1, \ldots, n\}}} (-1)^{|\ml{A}|} \ = \  \begin{cases} -1 & n=1 \\ 0 & n > 1. \end{cases}}$
\item For any $n$, $\displaystyle{\sum_{\substack{\ml{A}: A_1 \ = \  \eset, \\ A_k \ = \  \{1, \ldots, n\}}} (-1)^{|\ml{A}|} \ = \  (-1)^{n-1}.}$
\item For fixed $\ml{B}$ with $1 \in B_1$, $\displaystyle{\sum_{\substack{\ml{A} \succeq \ml{B}, 1 \in A_1, \\ A_k \ = \  \{1, \ldots, n\}}} (-1)^{|\ml{A}|} \ = \  \begin{cases} (-1)^n & \text{ {\rm if} } B_1 = \{1\}, \\ 0 & \text{ {\rm otherwise}.}\end{cases}}$
\end{enumerate}
\end{lemma}

(1) For $n=1$ there is only one possible chain of the desired form, namely $\{1\}$. Otherwise, there is a bijection between chains $\ml{A}$ of the desired form with $A_1 = \{1\}$ and those with $A_1 \supsetneq \{1\}$, by deleting or prepending $\{1\}$ from the beginning of the chain. Since this reverses the parity of $|\ml{A}|$, the sum vanishes. \\

(2) Inductively, consider a chain $\ml{A}' : \eset = A_1' \subset \cdots \subset A_k' = \{1, \ldots, n-1\}$ on $\{1, \ldots, n-1\}$. There are $t-1$ ways of inserting the element $n$ into the chain while keeping $A_1 = \eset$ and the last set equals $\{1, \ldots, n\}$: we can add it into one of the $A_i$, $i=2, \ldots, t$, or we can insert it immediately after $A_i$ as $A_i \cup \{n\}$, for $i = 1, \ldots, t$.

The chains $\ml{A}$ obtained this way contribute $(t-1) \cdot (-1)^{|\ml{A}'|} + t \cdot (-1)^{|\ml{A}'|+1} = (-1)^{|\ml{A}'|+1}$, giving the recurrence
\be
\sum_{\substack{\ml{A}: A_1 = \eset, \\ A_k = \{1, \ldots, n\}}} (-1)^{|\ml{A}|} \ = \  - \sum_{\substack{\ml{A}: A_1 = \eset, \\ A_k \ = \  \{1, \ldots, n-1\}}} (-1)^{|\ml{A}'|}.
\ee
For $n=1$, there is only one such chain, namely $\eset \subset \{1\}$, which has length 2.  \\

(3) Write $\ml{B} = B_1 \subset \cdots \subset B_k$. Choosing $\ml{A} \succeq \ml{B}$ is the same as choosing $k+1$ chains, namely, a chain with $1 \in A_1$ and $A_k = B_1$; then, for each $2 \leq i \leq k-1$, a chain from $B_i$ to $B_{i+1}$, and a chain from $B_k$ to $\{1, \ldots, n\}$. Thus, we factor our sum as
\be
\sum_{\substack{\ml{A} \succeq \ml{B}, 1 \in A_1, \\ A_k \ = \  \{1, \ldots, n\}}} (-1)^{|\ml{A}|} \ = \  (-1)^{|\ml{B}|}
\Big(\sum_{\substack{\ml{A} : 1 \in A_1 \\ A_k = B_1}} (-1)^{|\ml{A}|}\Big) \cdot \Big( \sum_{\substack{\ml{A} : A_1 = B_1 \\ A_k = B_2}}(-1)^{|\ml{A}|}\Big) \cdots \Big(\sum_{\substack{\ml{A} : A_1 = B_k \\ A_k = \{1, \ldots, n\}}} (-1)^{|\ml{A}|}\Big).
\ee
Each $B_i$ is double-counted in the lengths of the chains, so we multiply by $(-1)^k = (-1)^{|\ml{B}|}$.

By parts (1) and (2) above, this gives
\be
\ = \  (-1)^{|\ml{B}|} \left. \begin{cases} (-1) & B_1 =  \{1\} \\ 0 & B_1 \ne \{1\} \end{cases} \right\} \cdot (-1)^{|B_2| - |B_1|-1} \cdots (-1)^{n - |B_k|-1},
\ee
which is $(-1)^n$ when $B_1 = \{1\}$ and 0 otherwise, as desired.

\ \\

\end{document}